\documentclass[12pt]{amsart}
\usepackage[usenames,dvipsnames]{xcolor}

\usepackage[margin=1in]{geometry}

\usepackage{amsmath}
\usepackage{amsfonts}
\usepackage{mathrsfs}
\usepackage{amssymb}
\usepackage{hyperref,graphicx}
\usepackage{aligned-overset}
\usepackage{bm}
\usepackage{pinlabel}
\usepackage{tikz-cd}
\usepackage{enumitem}
\usepackage{multirow}

\newtheorem*{maintheorem*}{Main Theorem}

\newenvironment{mythm}[1]
  {\renewcommand\theinnercustomthm{#1}\innercustomthm}
  {\endinnercustomthm}

\newenvironment{mycorollary}[1]
  {\renewcommand\theinnercustomcorollary{#1}\innercustomcorollary}
  {\endinnercustomcorollary}

\newtheorem{theorem}{Theorem}[section]
\newtheorem{innercustomthm}{Theorem}
\newtheorem{lemma}[theorem]{Lemma}
\newtheorem{corollary}[theorem]{Corollary}
\newtheorem{innercustomcorollary}{Corollary}
\newtheorem{proposition}[theorem]{Proposition}

\theoremstyle{definition}
\newtheorem{definition}[theorem]{Definition}

\theoremstyle{remark}
\newtheorem{remark}[theorem]{Remark}

\numberwithin{equation}{section}

\newcommand{\scrC}{\mathscr{C}}

\newcommand{\Hom}{\mathrm{Hom}}

\newcommand{\Rom}[1]{\uppercase\expandafter{\romannumeral #1\relax}}

\newcommand{\hgr}{\text{gr}_{h}}
\newcommand{\qgr}{\text{gr}_{q}}
\newcommand{\fgr}{\text{gr}_{f}}

\newcommand{\Ob}{\text{Ob}}
\newcommand{\skspec}{\mathcal{X}_{Sk}}
\newcommand{\khspec}{\mathcal{X}_{Kh}}
\newcommand{\khspecgrad}[1]{\khspec^{#1}}
\newcommand{\skspecgrad}[2]{\skspec^{#1,#2}}
\newcommand{\calM}{\mathcal{M}}
\newcommand{\calF}{\mathcal{F}}

\title[Transverse Invariant in KSH]{Transverse invariant as Khovanov skein spectrum at its extreme Alexander grading}
\author{Nilangshu Bhattacharyya, Adithyan Pandikkadan}
\begin{document}
\maketitle

\begin{abstract}
We develop a space-level formulation of Khovanov skein homology by constructing a stable homotopy type for annular links. We explicitly define a cover functor from the skein flow category to the cube flow category, thereby establishing the Khovanov skein spectrum, $\mathcal{X}_{Sk}(L)$. This spectrum extends the framework of Lipshitz and Sarkar’s Khovanov spectrum $\mathcal{X}_{Kh}(L)$ \cite{LipSar2014} and provides new avenues for understanding transverse link invariants in the annular setting. Furthermore, we establish a map from the Khovanov spectrum to the skein spectrum, which, at extreme gradings, recovers the cohomotopy transverse invariant \cite{LipSarTransverse}.
\end{abstract}

\section{Introduction}\label{section 1}
Knots and links in contact manifolds play a fundamental role in low-dimensional topology and contact geometry. Understanding their classification and invariants provides deep insights into the structure of contact 3-manifolds. There are two special classes of knots in the standard contact 3-sphere, $(S^3,\xi_{{std}})$---Legendrian and transverse knots. Legendrian knots are everywhere tangent to the contact planes while transverse knots are everywhere transverse to them. 

We will restrict our attention to transverse links. There are two ``classical" invariants for transverse links---the topological link type and the \textit{self linking number}. While various knot types including the unknot, torus knots, and figure eight are completely classified by these invariants, such a classification is still unknown for most of the knot types. This calls for further introducing more refined invariants. 

The first candidate for such a non-classical invariant in the context of Khovanov homology called the \textit{transverse link invariant} was given by Olga Plamevskaya \cite{Olga2006}. Given the link diagram, $D_L$ of a transverse link, $L \subset (S^3,\xi_{std})$, the invariant is defined as a distinguished element, $\psi_{Kh}(L)$ in the Khovanov Homology \cite{OriginalKhovanovPaper, BarNatanKhovanov} of the link. While the effectiveness of this invariant remains an open question, other transverse invariants of different flavors have since been developed. Notably, Ozsváth, Szabó, and Thurston \cite{grid_invariant}, along with Lisca, Ozsváth, Stipsicz, and Szabó \cite{loss_invariant}, have constructed analogous invariants within the framework of knot Floer homology. Additionally, several refinements of the Plamenevskaya invariant have been developed, which could be more effective than the original invariant. Lipshitz and Sarkar \cite{LipSar2014} recently introduced a space-level version of Khovanov homology, called the Khovanov Spectrum. Later Lipschitz, Ng and Sarkar \cite{LipSarTransverse} showed that the transverse link invariant admits a refinement as a stable cohomotopy element of the Khovanov Spectrum which yields several other computable auxiliary invariants.

In our work, we consider links which can be represented as a braid closure of a transverse link. These can be naturally considered as embedded inside a thickened annulus, $L\subset A\times I \subset S^3$. M. M. Asaeda, J. H. Przytycki, and A. S. Sikora \cite{AsaedaPrzySikora2004} introduced a tri-graded cohomology for annular links,  called the Khovanov skein homology. Later, L. Roberts \cite{Rob2013}, gave an alternate formulation that directly relates this to Khovanov homology. There is a version of transverse invariant within this framework as well, denoted $\psi_{Sk}(L)$, which is a distinguished element in the Khovanov skein homology. Roberts also proved the existence of a link surgery spectral sequence from the Skein homology, $H^{h,q,f}_{Sk}(L)$ to knot Floer homology of the double-branched cover. A difference in skein homology is the additional grading, which arises from viewing the knot within the annulus. This grading corresponds to twice the Alexander grading induced by the braid axis in the knot Floer complex - hence the term ``extreme Alexander grading" in the title. Furthermore, Lawrence Roberts \cite{Rob2013} showed that under the spectral sequence, $\psi_{Sk}(L)$ is sent to a nonzero generator in the knot Floer complex. We aim to further explore this direction in future work.\par
In this paper, we develop a space-level formulation of Khovanov skein homology by explicitly defining a cover functor $\mathscr{F}$ from the skein flow category to the cube flow category, $\scrC_{Sk}(D_{L}) \to \scrC_{C}(n(D_{L}))$. A more detailed description of this construction can be found in Section \ref{section 3}. Utilizing this cover functor, we establish a stable homotopy type for annular links. It is worth noting that the annular Khovanov spectrum can also be constructed via an appropriate functor from the cube category $\underline{2}^{n}$ to the Burnside category $\mathscr{B}$; for further details on the Burnside category approach to the annular Khovanov spectrum, see \cite{QuantumAnnularSpectrum, SL2ActionOnAnnularSpectrum}. While our construction utilizing the Cohen-Jones-Segal realization of a framed flow category is initially distinct from the Burnside category approach, the work of Lawson, Lipshitz, and Sarkar \cite{BurnsideCatApproachTowardsKhovanovStableHomotopy} establishes that both spectra are, in fact, stably homotopy equivalent. Kauffman, Nikonov, and Ogasa \cite{SpectrumForKnotInThickenedHigherGenusSurface} have also defined a stable homotopy for the homotopical Khovanov homology for links in the thickened closed surfaces with genus $g>1$. The Khovanov skein spectrum we constructed in this paper is also very similar in spirit. 
\par
Lipshitz and Sarkar \cite{LipSar2014} constructed a spectrum for links, $L\subset S^{3}$ called the Khovanov Spectrum, $\khspec(L)$  such that the reduced singular cohomology, $\widetilde{H}^h(\khspec^q(L))$ is isomorphic to the Khovanov Homology, $H^{h,q}_{Kh}(L)$. 
Given a link diagram $L\subset S^2$ with $n$ crossings, one can associate a combinatorial cube flow category, denoted $\scrC_{C}(n)$. The objects of $\scrC_{C}(n)$ correspond to the $2^n$ possible states obtained by resolving each crossing in one of two ways, labeled as $0$- and $1$-resolutions, see Figure \ref{fig:resolution}. 
\begin{figure}[htp]
    \centering
    \tikzset{every picture/.style={line width=0.75pt}} 

\begin{tikzpicture}[x=0.75pt,y=0.75pt,yscale=-1,xscale=1]

\draw  [draw opacity=0] (271.52,74.42) .. controls (264.91,70.67) and (260.45,63.57) .. (260.45,55.43) .. controls (260.45,47.18) and (265.04,39.99) .. (271.8,36.29) -- (282.27,55.43) -- cycle ; \draw   (271.52,74.42) .. controls (264.91,70.67) and (260.45,63.57) .. (260.45,55.43) .. controls (260.45,47.18) and (265.04,39.99) .. (271.8,36.29) ;  
\draw  [draw opacity=0] (233.05,36.37) .. controls (239.89,40.19) and (244.45,47.58) .. (244.24,55.96) .. controls (244.06,63.7) and (239.86,70.41) .. (233.68,74.13) -- (222.43,55.43) -- cycle ; \draw   (233.05,36.37) .. controls (239.89,40.19) and (244.45,47.58) .. (244.24,55.96) .. controls (244.06,63.7) and (239.86,70.41) .. (233.68,74.13) ;  
\draw  [draw opacity=0] (388.51,74.09) .. controls (392.26,67.48) and (399.36,63.02) .. (407.5,63.02) .. controls (415.76,63.02) and (422.94,67.6) .. (426.65,74.37) -- (407.5,84.83) -- cycle ; \draw   (388.51,74.09) .. controls (392.26,67.48) and (399.36,63.02) .. (407.5,63.02) .. controls (415.76,63.02) and (422.94,67.6) .. (426.65,74.37) ;  
\draw  [draw opacity=0] (426.57,35.61) .. controls (422.75,42.46) and (415.36,47.01) .. (406.97,46.81) .. controls (399.23,46.62) and (392.52,42.42) .. (388.8,36.24) -- (407.5,25) -- cycle ; \draw   (426.57,35.61) .. controls (422.75,42.46) and (415.36,47.01) .. (406.97,46.81) .. controls (399.23,46.62) and (392.52,42.42) .. (388.8,36.24) ;  
\draw    (310.75,74.42) -- (349.22,36.37) ;
\draw    (310.75,36.37) -- (326.85,52.31) ;
\draw    (333.12,58.48) -- (349.22,74.42) ;
\draw    (281.84,55.03) -- (302.98,55.03) ;
\draw [shift={(279.84,55.03)}, rotate = 0] [color={rgb, 255:red, 0; green, 0; blue, 0 }  ][line width=0.75]    (10.93,-3.29) .. controls (6.95,-1.4) and (3.31,-0.3) .. (0,0) .. controls (3.31,0.3) and (6.95,1.4) .. (10.93,3.29)   ;
\draw    (357.54,55.03) -- (378.68,55.03) ;
\draw [shift={(380.68,55.03)}, rotate = 180] [color={rgb, 255:red, 0; green, 0; blue, 0 }  ][line width=0.75]    (10.93,-3.29) .. controls (6.95,-1.4) and (3.31,-0.3) .. (0,0) .. controls (3.31,0.3) and (6.95,1.4) .. (10.93,3.29)   ;

\draw (290.32,34.99) node [anchor=north west][inner sep=0.75pt]  [font=\footnotesize] [align=left] {$\displaystyle 0$};
\draw (359.02,35.99) node [anchor=north west][inner sep=0.75pt]  [font=\footnotesize] [align=left] {$\displaystyle 1$};

\end{tikzpicture}
    \caption{$0$- and $1$- resolutions}
    \label{fig:resolution}
\end{figure}
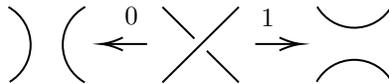

Each state $x \in \{0,1\}^{n}$ has a grading given by $\text{gr}(x)= \sum\limits_{i=1}^{n}x_{i}$. The morphisms are determined by the obvious partial order $1\succ 0$. For two distinct objects $x,y$ in $\scrC_{C}(n)$, if $x \succ y$, we define the set $\text{Hom}_{\scrC_{C}(n)}(x,y)$ to be a $\big(\text{gr}(x)-\text{gr}(y)-1\big)$-dimensional permutohedron which is a $\langle\text{gr}(x)-\text{gr}(y)-1\rangle$-manifold; otherwise, we declare it to be empty; see \cite[Definition 4.1]{LipSar2014}. We also define $\text{Hom}_{\scrC_{C}(n)}(x,x)$ to be a singleton consisting of the identity map. The cube flow category $\scrC_{C}(n)$ naturally arises as a Morse flow category corresponding to the Morse function $f_{n}:\mathbb{R}^{n} \to \mathbb{R}$, given by 
$$f_{n}(x_{1},\dots, x_{n}) = (3x_{1}^{2}-2x_{1}^3 )+ \cdots + (3x_{n}^{2}-2x_{n}^3).$$ 
Moreover, there exists a neat embedding $\imath$ and coherent framing $\psi$, which makes $\scrC_{C}(n)$ a framed flow category; see \cite[Proposition 4.12]{LipSar2014}. Lipshitz and Sarkar defined Khovanov flow category $\scrC_{Kh}(L)$ and built a cover functor $\mathscr{F}: \scrC_{Kh}(L) \to \scrC_{C}(n) $ for a choice called Ladybug matching; see \cite[Section 5]{LipSar2014} and also Section \ref{section 3}. The Khovanov flow category $\scrC_{Kh}(L)$ is a framed flow category where the framing is induced by the cover functor $\mathscr{F}$ from the framing of the cube flow category $\scrC_{C}(n)$. Taking the formal desuspension of Cohen-Jones-Segal realization $|\scrC_{Kh}|_{\widetilde{\imath},\widetilde{d}, \widetilde{\psi}}$ gives a suspension spectrum called the Khovanov spectrum $\khspec(L)$. Moreover, the stable homotopy type of $\khspec(L)$ depends only on the isotopy class of the corresponding link. Figure \ref{fig:overview of Khovanov spectra} presents a flowchart summarizing the construction of the stable homotopy type based on Cohen-Jones-Segal realization of a framed flow category.
\par
In the Khovanov skein chain complex, the key difference is that resolutions of $L$ include an additional marking that records the braid axis. So we are essentially viewing the braid diagram in the annulus $A$ rather than in $S^2$. Each resolution consists of a disjoint union of embedded circles, which can be classified as either trivial or non-trivial based on whether their homotopy classes in $\pi_{1}(A)$ are trivial or non-trivial, respectively. 

Following a similar framework as in Khovanov homology, we construct the Khovanov skein spectrum, $\skspec(L)$ such that the reduced singular cohomology, $\widetilde{H}^h(\skspec^{q,f}(L))$ is isomorphic to the Khovanov skein homology, $H^{h,q,f}_{Sk}(L)$.  We also show that the stable homotopy type of the Khovanov skein spectrum is defined up to the allowable isotopies of the link in the thickened annulus.\par

\begin{figure}[htp]
    \centering
    \input{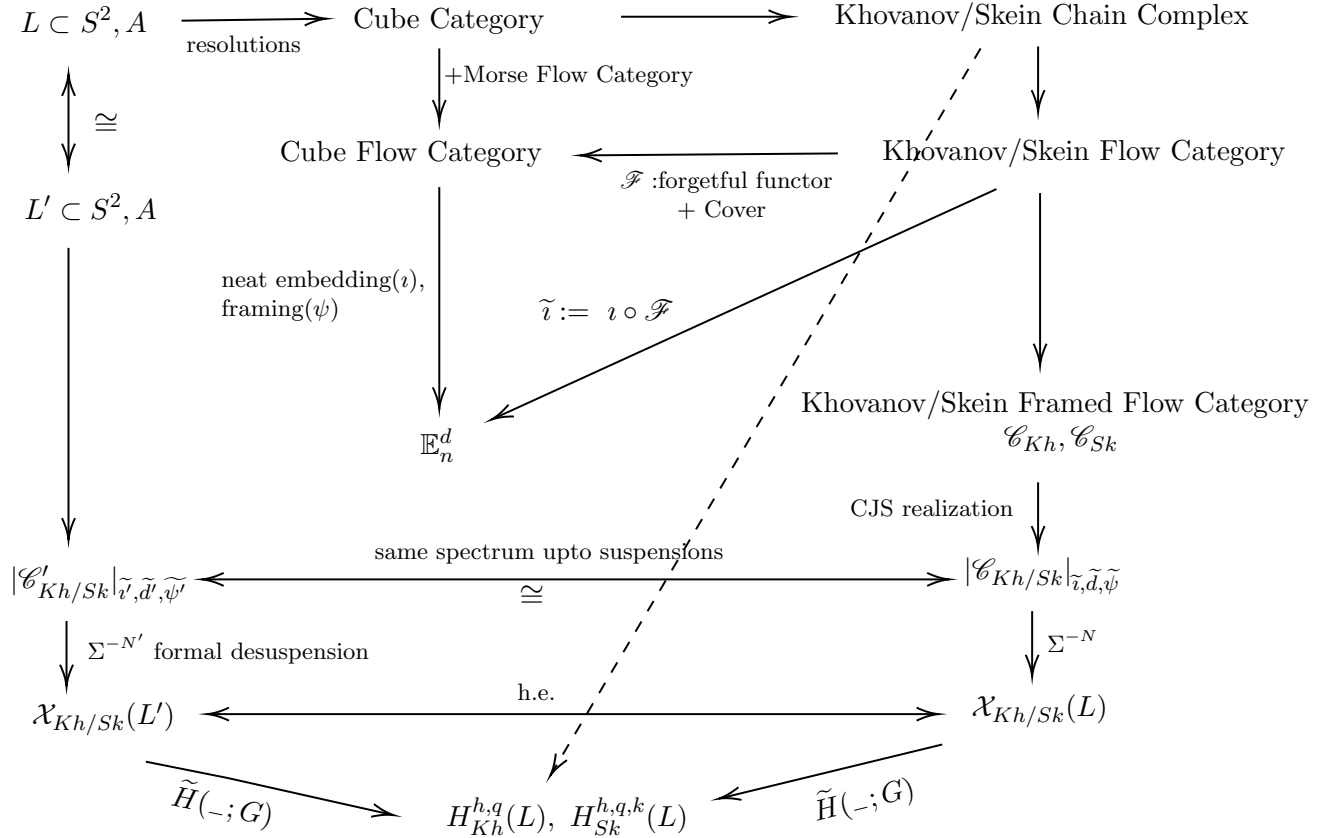}
    \caption{Flowchart describing the construction of Khovanov spectra/ Khovanov skein spectra}
    \label{fig:overview of Khovanov spectra}
\end{figure}

 \begin{mythm}{3.12}
     Let $L$ be a link diagram in annulus and $L^{'}$ is obtained from $L$ through allowable Reidemeister Moves $\Rom{1}$, $\Rom{2}$, and $\Rom{3}$ in the annulus then, $\skspecgrad{q}{f}(L)$ is stably homotopy equivalent to $\skspecgrad{q}{f}(L^{'})$. 
 \end{mythm}
For every quantum grading, $j$ we proved that there exists a map between the corresponding spectra, $\Psi^{j}(B_{L}):\khspecgrad{j}(B_{L}) \to \skspecgrad{j}{f_{\text{min}}(B_{L},j)}(B_{L})$, where $f_{\text{min}}(B_{L},j)$ is the lowest homotopical grading in the given quantum grading $j$. \par 
\begin{mythm}{4.1}
    For an oriented closed braid diagram $B_{L}\subset A$ with $b(B_{L})$ number of strands representing $L\subset A\times I$, and for each quantum grading $j$, there is a map 
    $$
    \Psi^{j}(B_{L}):\khspecgrad{j}(B_{L}) \to \skspecgrad{j}{f_{\text{min}}(B_{L},j)}(B_{L}),
    $$ 
    such that the induced map on the reduced cohomology 
    $$
    \Psi^{j}(B_{L})^{*}: \widetilde{H}^{i}\big(\skspecgrad{j}{f_{\text{min}} (B_{L},j)}(B_{L}); G\big) \to \widetilde{H}^{i}\big(\khspecgrad{j}(B_{L});G\big)
    $$
    is the same map as $\jmath^{*}: H_{Sk}^{i,j,f_{\text{min}} (B_{L},j)}(L;G) \to H_{Kh}^{i,j}(L;G)$ induced from the embedding $\jmath:  A \times I \hookrightarrow S^{3}$, for $G= \mathbb{Z}_{2}$ or $\mathbb{Z}$.
\end{mythm}
We further showed that at the quantum grading $j=sl(L)$, the extreme spectrum is the sphere spectrum, $\skspecgrad{sl(L)}{-b(B_{L})}(B_{L})= \mathbb{S}$, and we recover the cohomotopy Plamenevskaya invariant $\Psi^{sl(L)}(B_{L}): \khspecgrad{sl(L)}(B_{L}) \to \mathbb{S}$ defined by Lipshitz, Ng, and Sarkar \cite{LipSarTransverse}. This is a refinement of the original Plamenvskaya invariant \cite{Olga2006}. Lipshitz, Ng, and Sarkar \cite{LipSarTransverse} showed that this map $\Psi^{sl(L)}(B_{L}) \in \pi_{s}^{0}(\khspecgrad{sl(L)}(L))$ is a transverse invariant. \par
\begin{mycorollary}{4.4}
    For a closed braid diagram $B_{L}$ representing a transverse link $L$, the extreme Khovanov skein spectrum $\skspecgrad{j_{min}}{f_{\text{min}}(B_{L})}(B_{L}) = \mathbb{S}$ is the sphere spectrum.
\end{mycorollary}
The term extreme Alexander grading is a reference to the link surgery spectral sequence in \cite{Rob2013}, under which the homotopical grading corresponds to the additional Alexander grading in the knot Floer homology induced by the binding which is the braid axis. 
\section{Resolution Configuration, Skein Homology, Transverse Invariant}\label{section 2}
\subsection{Resolution configurations in annulus and Khovanov skein homology}
In this section we introduce the notion of resolution configuration in the annulus, following the framework of Lipshitz and Sarkar's \cite{LipSar2014} resolution configuration in $S^2$. The key difference is that, in our setting, the circles are embedded in an annulus rather than in $S^2$. This distinction introduces additional structure, including an extra grading known as the homotopical grading, which is related to the Alexander grading in the context of knot Floer homology.

\begin{definition}
    A \textit{resolution configuration $D$ in annulus $A$} is a pair $(Z(D), R(D))$, where $Z(D)$ is a collection of pairwise disjoint circles embedded in $A$, and $R(D)$ is an ordered collection of pairwise disjoint arcs embedded in $A$ such that each arc $R\in R(D)$ has its boundary $\partial R\subset \bigcup_{Z\in Z(D)}Z$. An arc $R$ is called an \textit{$m$-arc} if the two points in $\partial R$ lie in two different circles and a \textit{$\Delta$-arc} if both points lie in the same circle. We call a circle $Z\in Z(D)$, a \textit{trivial} circle if $Z$ bounds a disk in $A$, otherwise, we call it a \textit{non-trivial} circle.
\end{definition}
\begin{definition}[\cite{LipSar2014}, Definition 2.3]
    Given two resolution configurations $D$ and $E$ we can define another resolution configuration $D\backslash E$ by declaring $Z(D\backslash E)= Z(D)\backslash Z(E)$ and $R(D\backslash E)= \{R\in R(D)\,|\, \partial R\cap Z =\emptyset,\forall Z \in Z(E)\}$. We define $D\cap E$ to be $D\backslash (D\backslash E)$.
\end{definition}
\begin{definition}
From a resolution configuration $D$ and a subcollection of arcs $R'(D)\subseteq R(D)$, we get another resolution configuration $s_{R'(D)}(D)$ by doing surgery along all the arcs in $R'(D)$, see Figure \ref{fig:resolution_configuration_and_dual} and also refer to \cite[Definition 2.5]{LipSar2014} for the $S^2$ version. Given a resolution configuration $D$ in $A$, we define the dual resolution configuration $D^{*}$ by declaring $Z(D^{*})= Z\big(s_{R(D)}(D)\big)$, for each arc $R_{i}\in R(D)$, there is a dual arc $R_{|R(D)|-i+1}^{*}\in R(D^{*})$; see Figure \ref{fig:resolution_configuration_and_dual}. For simplicity, we will write $s(D)$ to denote $s_{R(D)}(D)$ from here on. 
\end{definition}
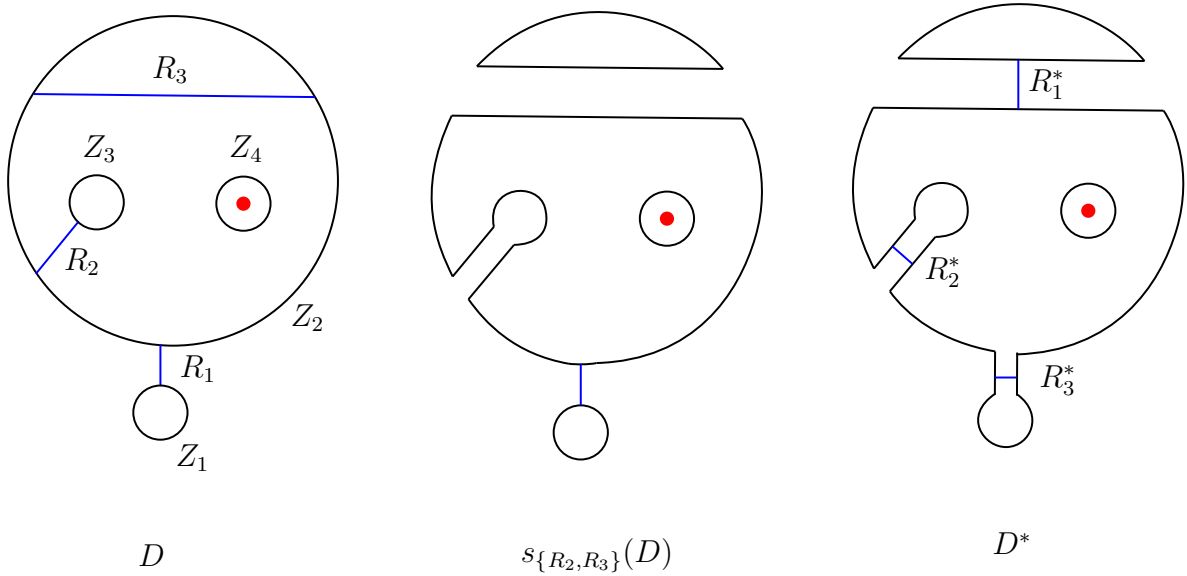
\begin{figure}[htp]
    \centering
    \tikzset{every picture/.style={line width=0.75pt}} 

\begin{tikzpicture}[x=0.75pt,y=0.75pt,yscale=-1,xscale=1]

\draw   (33,143.5) .. controls (33,97.57) and (70.23,60.33) .. (116.17,60.33) .. controls (162.1,60.33) and (199.33,97.57) .. (199.33,143.5) .. controls (199.33,189.43) and (162.1,226.67) .. (116.17,226.67) .. controls (70.23,226.67) and (33,189.43) .. (33,143.5) -- cycle ;
\draw   (64,154.33) .. controls (64,146.79) and (70.12,140.67) .. (77.67,140.67) .. controls (85.21,140.67) and (91.33,146.79) .. (91.33,154.33) .. controls (91.33,161.88) and (85.21,168) .. (77.67,168) .. controls (70.12,168) and (64,161.88) .. (64,154.33) -- cycle ;
\draw   (138,155) .. controls (138,147.45) and (144.12,141.33) .. (151.67,141.33) .. controls (159.21,141.33) and (165.33,147.45) .. (165.33,155) .. controls (165.33,162.55) and (159.21,168.67) .. (151.67,168.67) .. controls (144.12,168.67) and (138,162.55) .. (138,155) -- cycle ;
\draw [color=blue  ,draw opacity=1 ]   (46,99.67) -- (187.8,101.2) ;
\draw [color=blue  ,draw opacity=1 ]   (68.2,164.4) -- (47,190.4) ;
\draw [color=blue  ,draw opacity=1 ]   (109.8,226) -- (109.8,246.8) ;
\draw   (96.13,260.47) .. controls (96.13,252.92) and (102.25,246.8) .. (109.8,246.8) .. controls (117.35,246.8) and (123.47,252.92) .. (123.47,260.47) .. controls (123.47,268.01) and (117.35,274.13) .. (109.8,274.13) .. controls (102.25,274.13) and (96.13,268.01) .. (96.13,260.47) -- cycle ;
\draw  [color={rgb, 255:red, 252; green, 3; blue, 3 }  ,draw opacity=1 ][fill={rgb, 255:red, 252; green, 3; blue, 3 }  ,fill opacity=1 ][line width=0.75]  (154.74,155) .. controls (154.74,153.3) and (153.37,151.92) .. (151.67,151.92) .. controls (149.97,151.92) and (148.59,153.3) .. (148.59,155) .. controls (148.59,156.7) and (149.97,158.08) .. (151.67,158.08) .. controls (153.37,158.08) and (154.74,156.7) .. (154.74,155) -- cycle ;
\draw   (351.6,162.57) .. controls (351.6,155.02) and (357.72,148.9) .. (365.27,148.9) .. controls (372.81,148.9) and (378.93,155.02) .. (378.93,162.57) .. controls (378.93,170.11) and (372.81,176.23) .. (365.27,176.23) .. controls (357.72,176.23) and (351.6,170.11) .. (351.6,162.57) -- cycle ;
\draw    (256.71,110.57) -- (403.22,112.06) ;
\draw  [color={rgb, 255:red, 252; green, 3; blue, 3 }  ,draw opacity=1 ][fill={rgb, 255:red, 252; green, 3; blue, 3 }  ,fill opacity=1 ][line width=0.75]  (368.34,162.57) .. controls (368.34,160.87) and (366.97,159.49) .. (365.27,159.49) .. controls (363.57,159.49) and (362.19,160.87) .. (362.19,162.57) .. controls (362.19,164.27) and (363.57,165.64) .. (365.27,165.64) .. controls (366.97,165.64) and (368.34,164.27) .. (368.34,162.57) -- cycle ;
\draw    (269.39,85.67) -- (393.61,87) ;
\draw    (256.71,110.57) .. controls (237.73,150.5) and (249.91,176.68) .. (257.18,191.77) ;
\draw    (269.39,85.67) .. controls (310.32,42.82) and (367.41,55.73) .. (393.61,87) ;
\draw    (277.98,166.69) -- (257.18,191.77) ;
\draw    (288.16,175.78) -- (265.18,203.23) ;
\draw    (277.98,166.69) .. controls (273.91,145.59) and (300.27,143.59) .. (303.73,157.05) .. controls (307.18,170.5) and (298.23,175.71) .. (288.16,175.78) ;
\draw    (265.18,203.23) .. controls (273.67,215.25) and (288.38,230.98) .. (315.8,235.8) ;
\draw    (403.22,112.06) .. controls (408.58,119.4) and (418.33,143.08) .. (409.83,171.92) .. controls (401.33,200.75) and (379.34,233.95) .. (329.75,235.44) ;
\draw   (564.1,158.57) .. controls (564.1,151.02) and (570.22,144.9) .. (577.77,144.9) .. controls (585.31,144.9) and (591.43,151.02) .. (591.43,158.57) .. controls (591.43,166.11) and (585.31,172.23) .. (577.77,172.23) .. controls (570.22,172.23) and (564.1,166.11) .. (564.1,158.57) -- cycle ;
\draw    (469.21,106.57) -- (615.72,108.06) ;
\draw  [color={rgb, 255:red, 252; green, 3; blue, 3 }  ,draw opacity=1 ][fill={rgb, 255:red, 252; green, 3; blue, 3 }  ,fill opacity=1 ][line width=0.75]  (580.84,158.57) .. controls (580.84,156.87) and (579.47,155.49) .. (577.77,155.49) .. controls (576.07,155.49) and (574.69,156.87) .. (574.69,158.57) .. controls (574.69,160.27) and (576.07,161.64) .. (577.77,161.64) .. controls (579.47,161.64) and (580.84,160.27) .. (580.84,158.57) -- cycle ;
\draw    (481.89,81.67) -- (606.11,83) ;
\draw    (469.21,106.57) .. controls (450.23,146.5) and (462.41,172.68) .. (469.68,187.77) ;
\draw    (481.89,81.67) .. controls (522.82,38.82) and (579.91,51.73) .. (606.11,83) ;
\draw    (490.48,162.69) -- (469.68,187.77) ;
\draw    (500.66,171.78) -- (477.68,199.23) ;
\draw    (490.48,162.69) .. controls (486.41,141.59) and (512.77,139.59) .. (516.23,153.05) .. controls (519.68,166.5) and (510.73,171.71) .. (500.66,171.78) ;
\draw    (530.76,229.57) -- (530.64,251.36) ;
\draw    (541.9,230.14) -- (541.79,251.93) ;
\draw    (530.94,250.86) .. controls (513.9,262.15) and (525,278.35) .. (536.8,277.85) .. controls (548.6,277.35) and (555.8,261.55) .. (542.09,251.43) ;
\draw    (477.68,199.23) .. controls (486.17,211.25) and (503.33,224.75) .. (530.76,229.57) ;
\draw    (615.72,108.06) .. controls (621.08,115.4) and (630.83,139.08) .. (622.33,167.92) .. controls (613.83,196.75) and (591.09,229.59) .. (541.5,231.08) ;
\draw [color=blue  ,draw opacity=1 ]   (530.5,242.75) -- (541.33,242.75) ;
\draw [color=blue  ,draw opacity=1 ]   (478.83,176.42) -- (489.17,185.51) ;
\draw [color=blue  ,draw opacity=1 ]   (542.47,82.33) -- (542.47,107.31) ;
\draw    (315.8,235.8) .. controls (318.88,236.19) and (325.75,236.31) .. (329.75,235.44) ;
\draw   (308.13,270.47) .. controls (308.13,262.92) and (314.25,256.8) .. (321.8,256.8) .. controls (329.35,256.8) and (335.47,262.92) .. (335.47,270.47) .. controls (335.47,278.01) and (329.35,284.13) .. (321.8,284.13) .. controls (314.25,284.13) and (308.13,278.01) .. (308.13,270.47) -- cycle ;
\draw [color=blue  ,draw opacity=1 ]   (321.8,236) -- (321.8,256.8) ;

\draw (118.17,230.07) node [anchor=north west][inner sep=0.75pt]    {$R_{1}$};
\draw (59.67,176.07) node [anchor=north west][inner sep=0.75pt]    {$R_{2}$};
\draw (104.07,79.17) node [anchor=north west][inner sep=0.75pt]    {$R_{3}$};
\draw (116,274.9) node [anchor=north west][inner sep=0.75pt]    {$Z_{1}$};
\draw (174,204.4) node [anchor=north west][inner sep=0.75pt]    {$Z_{2}$};
\draw (69,118.9) node [anchor=north west][inner sep=0.75pt]    {$Z_{3}$};
\draw (143,118.9) node [anchor=north west][inner sep=0.75pt]    {$Z_{4}$};
\draw (551.67,235.07) node [anchor=north west][inner sep=0.75pt]    {$R_{3}^{*}$};
\draw (493.67,180.57) node [anchor=north west][inner sep=0.75pt]    {$R_{2}^{*}$};
\draw (546,85.73) node [anchor=north west][inner sep=0.75pt]    {$R_{1}^{*}$};
\draw (97,325.4) node [anchor=north west][inner sep=0.75pt]    {$D$};
\draw (290,322.4) node [anchor=north west][inner sep=0.75pt]    {$s_{\{}{}_{R_{2} ,R_{3}\}}( D)$};
\draw (528,319.4) node [anchor=north west][inner sep=0.75pt]    {$D^{*}$};

\end{tikzpicture}
    \caption{Some notations for resolution configuration. $D$ is the initial resolution configuration; $s_{\{R_2,R_3\}}(D)$ is the resolution configuration after surgery along arcs $R_2$ and $R_3$; $D^*$ is the dual resolution configuration.}
    \label{fig:resolution_configuration_and_dual}
\end{figure}

\begin{definition}
A \textit{labeled resolution configuration} $(D,x)$ is a resolution configuration $D$ in $A$ with a labeling $x$ on each circle in $Z(D)$ as follows. Each non-trivial circle is labeled by an element in $\{v_{+},v_{-}\}$, and each trivial circle is labeled by an element in $\{w_{+},w_{-}\}$. Given two labeled resolution configurations $(D,y)$ and $(E,x)$, we declare that $(D,y)\prec' (E,x)$ if the following holds:
\begin{enumerate}
    \item The induced labelings by $x$ and $y$ on $D\cap E$ are the same.
    \item $E$ is obtained from $D$ by surgery along an $m$-arc. So $D\backslash E$ contains two circles, say $Z_{i}$ and $Z_{j}$, and $s(D\backslash E)= E\backslash D$ contains exactly one circle, say $Z_{k}$. There are three possibilities up to switching the role of $Z_{i}$ and $Z_{j}$:
        \begin{enumerate}
            \item \label{2a} The circles $Z_{i}$, $Z_{j}$ and $Z_{k}$ are all trivial circles.
            \item \label{2b} One of the circle in $\{Z_{i}, Z_{j}\}$ is trivial circle and the other circle is nontrivial. Without loss of generality, we assume that $Z_{i}$ is trivial, and $Z_{j}$ is nontrivial. In this case, $Z_{k}$ has to be nontrivial.
            \item \label{2c} Both circles $Z_{i}$ and $Z_{j}$ are nontrivial but then $Z_{k}$ has to be trivial. 
        \end{enumerate}
    \item $E$ is obtained from $D$ by surgery along a $\Delta$-arc. So $D\backslash E$ contains exactly circle, say $Z_{i}$, and $s(D\backslash E)= E\backslash D$ contains exactly two circles, say $Z_{j}$, and $Z_{k}$. There are three possibilities up to switching the role of $Z_{j}$ and $Z_{k}$:
    \begin{enumerate}
        \item \label{3a} The circles $Z_{i}$, $Z_{j}$ and $Z_{k}$ are all trivial circles.
        \item \label{3b} The circle $Z_{i}$ is nontrivial and one of the circle in $\{Z_{j}, Z_{k}\}$ is trivial and the other circle is nontrivial. Without loss of generality we assume that, the circle $Z_{j}$ is trivial and the circle $Z_{k}$ is nontrivial. 
        \item \label{3c}The circle $Z_{i}$ is trivial and the both circles in $\{Z_{j}, Z_{k}\}$ are nontrivial.
    \end{enumerate}
    \item In Case \ref{2a}, either $y(Z_{i})=y(Z_{j})=x(Z_{k})=w_{+}$, or $\{y(Z_{i}),y(Z_{j})\}=\{w_{+},w_{-}\}$ and $x(Z_{k})=w_{-}$. In Case \ref{2b}, either $y(Z_{i})=w_{+}$, and $y(Z_{j})= v_{+}$, and $x(Z_{k})=v_{+}$ or, $y(Z_{i})=w_{+}$, and $y(Z_{j})= v_{-}$, and $x(Z_{k})=v_{-}$. In Case \ref{2c}, $\{y(Z_{i}),y(Z_{j})\}=\{v_{+},v_{-}\}$ and $x(Z_{k})= w_{-}$.\par
    In Case \ref{3a}, either $y(Z_{i})=x(Z_{j})=x(Z_{k})=w_{-}$, or $y(Z_{i})=w_{+}$, and $\{x(Z_{j}),x(Z_{k})\}=\{w_{+},w_{-}\}$. In Case \ref{3b},  either $y(Z_{i})=x(Z_{k})= v_{+}$, and $x(Z_{j})=w_{-}$, or $y(Z_{i})= x(Z_{k})= v_{-}$, and $x(Z_{j})=w_{-}$. In Case \ref{3c}, $y(Z_{i})= w_{+}$ and $\{x(Z_{j}),x(Z_{k})\}=\{v_{+},v_{-}\}$.
\end{enumerate}
Now we define $\prec'$ to be the transitive closure of the above relation.
\end{definition}

\begin{definition}
A \textit{decorated resolution configuration in $A$ } is a triple $(D,x,y)$ where $(D,y)$ and $(s
(D),x)$ are labeled resolution configurations in $A$ satisfying: $(D,y)\prec' (s(D),x)$; see also \cite[Definition 2.11]{LipSar2014}. Associated to each decorated resolution configuration $(D,x,y)$, there is a poset $P(D,x,y)$ consisting of all labeled resolution configurations $(E,z)$ satisfying: $(D,y)\preceq' (E,z) \preceq' (s(D),x)$. 
\end{definition}
\begin{definition}[\cite{LipSar2014}, Definition 2.12]
Given a decorated resolution configuration $(D,x,y)$, we have a dual \textit{decorated resolution configuration}$(D^{*},y^{*},x^{*})$, where $D^{*}$ is dual of $D$ and $x^{*}$ is the labeling on $D^{*}$ dual to $x$, i.e. on each circle in $Z(D^{*})=Z(s(D))$, the labeling $x$ and $x^{*}$ are different, similarly $y$ and $y^{*}$ are dual labelings on circles in $Z(D)= Z(s(D^{*}))$. The poset $P(D^{*},y^{*},x^{*})$ is the reverse of the poset $P(D,x,y)$; see \cite[Lemma 2.13]{LipSar2014}. 
\end{definition}
The Khovanov skein homology of an oriented link $L\subset A\times I$ was originally defined by M. M. Asaeda, J. H. Przytycki, and A. S. Sikora  \cite{AsaedaPrzySikora2004} but we will use the following version defined by L. Roberts in \cite{Rob2013} as this definition directly relates to Khovanov homology.

\begin{definition}
Let $D_{L}$ be a link diagram in $A$ corresponding to an oriented link $L\subset A\times I$. We choose an ordering $1,\dots,n(D_{L})$ of the crossings of $D_{L}$, where $n(D_{L})$ denotes the total number of crossings in $D_{L}$. Let $n_{-}(D_{L})$ (resp. $n_{+}(D_{L})$) denote the number of negative crossings (resp. positive crossings) in $D_{L}$. For each state $\alpha \in \{0,1\}^{n(D_{L})}$, we assign a resolution configuration $D_{L}(\alpha)$ by resolving the $i$-th crossing by $0$-resolution (resp. $1$-resolution) if $\alpha_{i}=0$ (resp. $\alpha_{i}=1$), and then we place an arc at the $i$-th crossing if $\alpha_{i}=0$; see Figure \ref{fig:resolution_configuration_for_state}.
\begin{figure}[htp]
    \centering
    \tikzset{every picture/.style={line width=0.75pt}} 

\begin{tikzpicture}[x=0.75pt,y=0.75pt,yscale=-1,xscale=1]

\draw    (170.56,60.39) -- (191.22,82.17) ;
\draw    (168.33,84.39) -- (177.67,73.94) ;
\draw    (183,68) -- (190.33,58.61) ;
\draw    (168.33,84.39) -- (190.33,105.28) ;
\draw    (181.89,92.39) -- (191.22,82.17) ;
\draw    (167.67,104.83) -- (175.67,97.72) ;
\draw    (167.67,104.83) -- (187.22,124.61) ;
\draw    (190.33,105.28) -- (181.89,114.17) ;
\draw    (175.89,118.83) -- (168.78,124.61) ;
\draw    (170.56,60.39) .. controls (171,1.28) and (246.78,11.28) .. (243.89,59.94) ;
\draw    (190.33,58.61) .. controls (191,30.75) and (223.67,35.5) .. (224.56,57.72) ;
\draw    (224.56,57.72) -- (225.44,126.17) ;
\draw    (243.89,59.94) -- (245,126.61) ;
\draw    (187.22,124.61) .. controls (188.56,152.17) and (224.78,150.17) .. (225.44,126.17) ;
\draw    (168.78,124.61) .. controls (167,163.75) and (239,188.25) .. (245,126.61) ;
\draw  [color={rgb, 255:red, 252; green, 3; blue, 3 }  ,draw opacity=1 ][fill={rgb, 255:red, 252; green, 3; blue, 3 }  ,fill opacity=1 ][line width=0.75]  (209.24,94.5) .. controls (209.24,92.8) and (207.87,91.42) .. (206.17,91.42) .. controls (204.47,91.42) and (203.09,92.8) .. (203.09,94.5) .. controls (203.09,96.2) and (204.47,97.58) .. (206.17,97.58) .. controls (207.87,97.58) and (209.24,96.2) .. (209.24,94.5) -- cycle ;
\draw    (410.56,56.89) -- (410,79.75) ;
\draw    (427,102.47) -- (427.22,121.11) ;
\draw    (409.4,102.07) -- (408.78,121.11) ;
\draw    (410.56,56.89) .. controls (411,-2.22) and (486.78,7.78) .. (483.89,56.44) ;
\draw    (430.33,55.11) .. controls (431,27.25) and (463.67,32) .. (464.56,54.22) ;
\draw    (464.56,54.22) -- (465.44,122.67) ;
\draw    (483.89,56.44) -- (485,123.11) ;
\draw    (427.22,121.11) .. controls (428.56,148.67) and (464.78,146.67) .. (465.44,122.67) ;
\draw    (408.78,121.11) .. controls (407,160.25) and (479,184.75) .. (485,123.11) ;
\draw  [color={rgb, 255:red, 252; green, 3; blue, 3 }  ,draw opacity=1 ][fill={rgb, 255:red, 252; green, 3; blue, 3 }  ,fill opacity=1 ][line width=0.75]  (449.24,91) .. controls (449.24,89.3) and (447.87,87.92) .. (446.17,87.92) .. controls (444.47,87.92) and (443.09,89.3) .. (443.09,91) .. controls (443.09,92.7) and (444.47,94.08) .. (446.17,94.08) .. controls (447.87,94.08) and (449.24,92.7) .. (449.24,91) -- cycle ;
\draw    (410,79.75) .. controls (407,88.88) and (429.75,91.63) .. (429.5,80.75) ;
\draw    (409.4,102.07) .. controls (409.25,93.88) and (427.75,95.13) .. (427,102.47) ;
\draw    (430.33,55.11) -- (429.5,80.75) ;
\draw [color=blue  ,draw opacity=1 ]   (419.08,87.75) -- (419,96.69) ;

\draw (149,61.4) node [anchor=north west][inner sep=0.75pt]    {$1$};
\draw (149,85.4) node [anchor=north west][inner sep=0.75pt]    {$2$};
\draw (148.5,108.9) node [anchor=north west][inner sep=0.75pt]    {$3$};
\draw (199.4,173) node [anchor=north west][inner sep=0.75pt]    {$D_{L}$};
\draw (403.87,171.63) node [anchor=north west][inner sep=0.75pt]    {$D_{L}(( 1,0,1))$};
\draw (389,57.9) node [anchor=north west][inner sep=0.75pt]    {$1$};
\draw (389.67,82.83) node [anchor=north west][inner sep=0.75pt]    {$2$};
\draw (388.5,105.4) node [anchor=north west][inner sep=0.75pt]    {$3$};

\end{tikzpicture}
    \caption{Resolution configurations for a knot diagram.}
    \label{fig:resolution_configuration_for_state}
\end{figure}
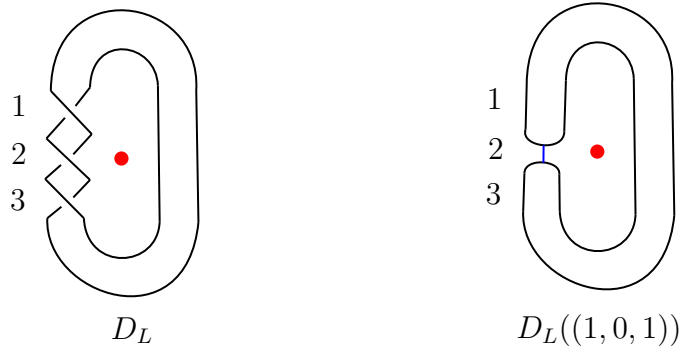
The (Manhattan) norm $ \sum_{i=1}^{n(D_{L})}\alpha_{i}$ is denoted by $|\alpha|$. Corresponding to each labeled resolution configuration $(D_{L}(\alpha),y)$ in $A$, we call the tuple $(\alpha,y)$ as \textit{enhanced state of $D_{L}$ in $A$ }. For each $(D_{L}(\alpha),y)$, we define three gradings as follows:
\begin{align*}
    \hgr((D_{L}(\alpha),y)) &= -n_{-} + |\alpha|, \\
    \qgr((D_{L}(\alpha),y)) &= n_{+}-2n_{-} + |\alpha| + \big|\{Z\in Z(D_{L}(\alpha))\,|\, y(Z)= v_{+},\,\text{or}\,\, y(Z)=w_{+}\}\big|\\
         & \qquad \qquad \qquad \qquad - \big|\{Z\in Z(D_{L}(\alpha))\,|\, y(Z)= v_{-},\,\text{or}\,\, y(Z)=w_{-}\}\big|,\\
    \fgr((D_{L}(\alpha),y)) &= \big|\{Z\in Z(D_{L}(\alpha))\,|\, y(Z)= v_{+}\}\big| 
                            - \big|\{Z\in Z(D_{L}(\alpha))\,|\, y(Z)= v_{-}\}\big|.
\end{align*}
The gradings $\hgr$, $\qgr$ and $\fgr$ are called \textit{homological grading} (or sometimes $h$-grading), \textit{quantum grading} (or sometimes $q$-grading) and \textit{homotopical grading} (or sometimes $f$-grading) respectively. 
\end{definition}
The cochain group $C_{Sk}(D_{L})$ is defined to be the $\mathbb{Z}$-module freely generated by labeled resolution configurations $(D_{L}(\alpha),y)$ in $A$. The boundary maps $\partial_{Sk}$ are defined in the following way. 
\begin{align*}
    \partial_{Sk}\big((D_{L}(\alpha),y)\big)= \sum\limits_{\substack{(D_{L}(\beta),x)\\ |\beta|=|\alpha|+1 \\ (D_{L}(\alpha),y)\prec' (D_{L}(\beta),x)}}  (-1)^{s_{0}(\mathcal{C}_{\beta,\alpha})} (D_{L}(\beta),x),
\end{align*}
where $s_{0}(\mathcal{C}_{\beta,\alpha}) \in \mathbb{Z}_{2}$ is called the \textit{standard sign assignment} which is defined as follows. Since $s(D_{L}(\alpha),y)\prec' (D_{L}(\beta),x)$ and $|\beta|=|\alpha|+1$, there is a unique $i\in \{1,\dots,n(D_{L})\}$ such that $\alpha_{i}=0$ but $\beta_{i}=1$. We declare $s_{0}(\mathcal{C}_{\beta,\alpha}) = \overline{\sum_{l=1}^{i-1}\alpha_{l}} \in \mathbb{Z}_{2}$.
Observe that the boundary maps preserve the quantum grading and homotopical grading but increase the homological grading by 1. It is known that $\partial_{Sk}\circ \partial_{Sk}=0$ and thus we get a tri-graded cohomology $H^{i,j,k}_{Sk}(L)$ of the complex $(C_{Sk}(D_{L}),\partial_{Sk})$, where $i$, $j$  and $k$ are homological, quantum, and homotopical grading respectively. The tri-graded cohomology $H^{i,j,k}_{Sk}(L)$ is invariant of the oriented link $L\subset A\times I$; see \cite[Theorem 2.1]{Rob2013}, and \cite[Theorem 2.1]{Yi-Xie-instantons-annular-khovanov}. We can similarly define a tri-graded cohomology $H^{i,j,k}_{Sk}(L;\mathbb{Z}_{2})$ with $\mathbb{Z}_{2}$ coefficient of the cochain complex $(\overline{C_{Sk}(D_{L})},\overline{\partial_{Sk}};\mathbb{Z}_{2})$ in a similar way, where the cochain group $\overline{C_{Sk}(D_{L})}$ is defined to be the $\mathbb{Z}_{2}$-vector space generated by the labeled resolution configurations $(D_{L}(\alpha),y)$, and the boundary $\overline{\partial_{Sk}}$ is defined as follows:
\begin{align*}
    \overline{\partial_{Sk}}\big((D_{L}(\alpha),y)\big)= \sum\limits_{\substack{(D_{L}(\beta),x)\\ |\beta|=|\alpha|+1 \\ (D_{L}(\alpha),y)\prec' (D_{L}(\beta),x)}}  (D_{L}(\beta),x).
\end{align*}
Observe that the tri-graded cohomology $H^{i,j,k}_{Sk}(L;\mathbb{Z}_{2})$ is the same as the \textit{Khovanov Skein homology} described in \cite[Section 2]{Rob2013}.\par
For an oriented link diagram $D_{L}\subset A$, and any group $G$, we define the following:
\begin{align*}
    &f_{\text{min}}(D_{L},j)= \min \big\{\fgr((D_{L}(\alpha),y))\,|\, (D_{L}(\alpha),y) \in C_{Sk}(D_{L})\,:\, \qgr((D_{L}(\alpha),y))=j\big\},\\
    &f_{\text{min}}^{H}(D_{L},j;G)=\min\big\{\fgr((D_{L}(\alpha),y))\,|\, \big[(D_{L}(\alpha),y)\big] \neq 0 \in H_{Sk}(D_{L};G)\,:\, \qgr((D_{L}(\alpha),y))=j\big\},\\
    &f_{\text{min}}(D_{L}) = \min\big\{\fgr((D_{L}(\alpha),y))\,|\, (D_{L}(\alpha),y) \in C_{Sk}(D_{L})\big\},\\
    &f_{\text{min}}^{H}(D_{L};G)=\min \big\{\fgr((D_{L}(\alpha),y))\,|\, \big[(D_{L}(\alpha),y)\big] \neq 0 \in H_{Sk}(D_{L};G)\big\},\\
    &S_{j,k}(D_{L})= \big\{(D_{L}(\alpha),y) \in C_{Sk}(D_{L})\,:\, \qgr((D_{L}(\alpha),y))=j,\,\text{and}\,\,\fgr((D_{L}(\alpha),y))=k \big\},\\
    &S_{k}(D_{L})= \big\{(D_{L}(\alpha),y) \in C_{Sk}(D_{L})\,:\, \fgr((D_{L}(\alpha),y))=k \big\}.
\end{align*}
\begin{remark}
    Since $H_{Sk}^{i,j,k}(L)$ is an invariant of oriented link $L\subset A\times I$, if two link diagrams $D_{L}$ and $D_{L}'$ in $A$ are related by allowable Reidemeister Moves $\Rom{1}$, $\Rom{2}$, and $\Rom{3}$ in annulus, then for each quantum grading $j$, $f_{\text{min}}^{H}(D_{L},j;G)= f_{\text{min}}^{H}(D_{L}',j;G)$, and $f_{\text{min}}^{H}(D_{L};G)= f_{\text{min}}^{H}(D_{L}';G)$. 
\end{remark}
Notice that we define the Khovanov skein homology in the same spirit Robert Lipshitz, and Sucharit Sarkar defined Khovanov cochain complex $(C_{Kh}^{i,j}(D_{L}), \partial_{Kh})$ for an oriented link diagram $D_{L}\subset S^{2}$ representing an oriented link $L$ in $S^{3}$ in \cite[Definition 2.15]{LipSar2014}, where $i$ and $j$ denote the homological and quantum grading respectively. In \cite{LipSar2014}, Robert Lipshitz, and Sucharit Sarkar used a partial order $\prec$ on the labeled resolution configurations in $S^{2}$ to define the boundary maps. Associated to a state $\alpha$, we get labelled resolution configurations of the form $(D_{L}(\alpha),y)$, where $D_{L}(\alpha)$ is a resolution configuration in $S^{2}$, with a labeling $y$ on the circles in $Z(D_{L}(\alpha))$ where on each circle $y$ is either $x_{+}$, or $x_{-}$; see \cite[Section 2]{LipSar2014}. We then define the bigraded Khovanov homology $H_{Kh}^{i,j}(L)$ of $L$ as the homology of the cochain complex $C_{Kh}^{h,q}(D_{L})$.

\subsection{Transverse invariant in Khovanov Skein homology} Consider $S^3$ with standard contact structure $\xi_{std}=\ker(dz-ydx)$. A link in $S^{3}$ is called a \textit{transverse link} if it is everywhere transverse to the standard contact planes. By a theorem of D. Bennequin \cite{Benn1983}, any transverse link in $(S^{3}, \xi_{sym}= \ker(dz-ydx+xdy)) $ is transversely isotopic to a closed braid. Moreover, by the transverse Markov theorem two closed braids $L$ and $L'$ represent transversely isotopic links if and only if $L$ and $L'$ are related by braid conjugation, braid isotopies, and positive braid stabilization and destabilization; see \cite{Orevkov-Shevchishin-transverse-markov}.\par
Given a transverse link $L$, Olga Plamenevskaya \cite{Olga2006} introduced a transverse invariant $\psi_{Kh}(L)$ defined as a distinguished element in Khovanov homology, $H_{Kh}(L)$. Let $D_{L}$ be a closed braid diagram which represents $L$ with $b$ strands. Moreover, let $n(D_{L})$, $n_{+}(D_{L})$, and $n_{-}(D_{L})$ denote the total number of crossings, the number of positive crossings, and the number of negative crossings in $D_{L}$ respectively. Then the \textit{self-linking number} of $L$ is given by 
$$
sl(L)=-b + n_{+}(D_{L}) - n_{-}(D_{L}).
$$ 
The \textit{oriented resolution} $D_{L}^{o}= D_{L}(\alpha^{o})$ as described in \cite{LipSarTransverse, Olga2006, Olga2006} corresponds to the state $\alpha^{o}\in \{0,1\}^{n(D_{L})}$, where $\alpha^{o}_{i}=1$ if $i$-th crossing is a negative crossing and $\alpha^{o}_{i}=0$ otherwise. Note that the number of circles $|Z(D_{L}^{o})|=b$. 

Consider the labeled resolution configuration $(D_{L}^{o},x^{o})$ where each circle in $Z(D_{L}^{o})$ is assigned the label $x_{-}$. This defines a generator $\widehat{\psi}_{Kh}(L)\in C_{Kh}^{0,sl(L)}(D_{L})$, which is a cocycle and thus corresponds to an element $\psi_{Kh}(L)\in H^{0,sl(L)}_{Kh}(L)$. Plamenevskaya's work \cite[Theorem 2]{Olga2006} shows that up to a sign, $\psi_{Kh}(L)$ is a transverse link invariant.\par 
This transverse invariant can also be viewed as an element $\widetilde{\psi}(L)$ in the reduced Khovanov homology $\widetilde{H}_{Kh}(L)$; see \cite[Section 2]{Olga2006_2}. The reduced Khovanov homology $\widetilde{H}_{Kh}^{i,j}(L)$ in $\mathbb{Z}$-coefficient for an oriented link $L$ with a marked point is the bi-graded homology of the reduced Khovanov cochain complex of $\widetilde{C}_{Kh}^{i,j}(D_{L})$, where $D_{L}$ is a link diagram representing $L$ with the marked point. The reduced Khovanov cochain complex $\widetilde{C}_{Kh}(D_{L})$ is defined in the following way. There is a sub-cochain complex $C_{Kh}^{x_{-}}(D_{L})$ of the Khovanov cochain complex defined as the $\mathbb{Z}$-module freely generated by those enhanced states of $D_{L}$ in $S^{2}$ where the marked circles are labeled with $x_{-}$. Now we define $\widetilde{C}_{Kh}(D_{L})$ to be $C_{Kh}(D_{L})/C_{Kh}^{x_{-}}(D_{L})$. Note that there is an obvious isomorphism $\widetilde{C}_{Kh}(D_{L})\cong C_{Kh}^{x_{-}}(D_{L})$, which corresponds an enhanced state of $D_{L}$ in $S^{2}$ where the marked circle is labeled with $x_{+}$ to an enhanced state of $D_{L}$ in $S^{2}$ with marked circle labeled with $x_{-}$, by switching the labeling only on the marked circle. Note that $\widehat{\psi}(L)\in C_{Kh}^{x_{-}}(D_{L})$, but under the above isomorphism we can view $\widehat{\psi}(L)$ as an element of $\widetilde{C}_{Kh}(D_{L})$. This gives rise to a version of the transverse invariant $\widetilde{\psi}(L)\in \widetilde{H}_{Kh}(L)$; see \cite[Section 2]{Olga2006_2}.\par
The closed braid diagram $D_{L}$ representing the transverse link $L$ can be viewed as a closed braid diagram embedded in an annulus $A$. Thus the oriented resolution $D_{L}^{o}$ can be thought of as a resolution configuration in $A$ where all circles in $Z(D_{L}^{o})$ are non-trivial. 

Now, consider the labeled resolution configuration $(D_{L}^{o},x^{o}_{Sk})$ where each circle in $Z(D_L^o)$ is assigned the label $v_-$. This defines an element $\widehat{\psi}_{Sk}(D_{L})\in C_{Sk}^{0,sl(L),-b}(D_{L})$. 
It is straightforward to verify that $\widehat{\psi}_{Sk}(D_{L})$ is a closed element in $C_{Sk}(D_{L})$, representing a homology class $\psi_{Sk}(D_{L})$ with $\fgr(\psi_{Sk}(D_{L}))= f_{\text{min}}(D_{L})$. Moreover, $\widehat{\psi}_{Sk}(D_{L})$ is the unique enhanced state of $D_{L}$ in $A$ with homotopical grading $f_{\text{min}}(D_{L})$, and therefore does not vanish in the $H_{Sk}(D_{L})$, see \cite[Theorem 7.1]{Rob2013}. Hence we have the following lemma.\par

\begin{lemma}\label{lemma 2.2}
    Let \( D_L \subset A \) be a closed braid diagram with \( b \) strands, whose braid axis coincides with the core of the annulus, representing a transverse link \( L \). Then the extreme homotopical grading satisfies
    $f_{\min}(D_L) = f_{\min}(D_L, sl(L)) = f_{\min}^{H}(D_L, sl(L); G) = f_{\min}^{H}(D_L; G) = -b$ for coefficient group \( G = \mathbb{Z}\) or $\mathbb{Z}_{2}$. Moreover, $S_{-b}(D_L) = S_{sl(L), -b}(D_L) = \{ \widehat{\psi}_{Sk}(D_L) \}.$
\end{lemma}

\begin{theorem}\label{prop 2.2}
    Let $D_{L}$ and $D_{L}'$ be two closed braid representatives of a transverse knot $L$, such that $D_{L}'$ is obtained from $D_{L}$ by a positive braid stabilization (or, $D_{L}$ is obtained from $D_{L}'$ by positive destabilization). Then there are cochain maps $\phi_{s}: C_{Sk}^{i,j,k}(D_{L})\to C_{Sk}^{i,j,k-1}(D_{L}')$, and $\phi_{d}: C_{Sk}^{i,j,k}(D'_{L}) \to C_{Sk}^{i,j,k+1}(D'_{L})$, such that they give isomorphisms: $\phi_{s}: C^{0,sl(L),f_{\text{min}}(D_{L})}_{Sk}(D_{L})\to$ $C^{0,sl(L),f_{\text{min}}(D_{L}')}_{Sk}(D_{L}')$, and $\phi_{d}: C^{0,sl(L),f_{\text{min}}(D_{L}')}_{Sk}(D_{L}') \to C^{0,sl(L),f_{\text{min}}(D_{L})}_{Sk}(D_{L})$ which are inverse to each other, and satisfy: $\phi_{s}(\widehat{\psi}_{Sk}(D_{L}))= \widehat{\psi}_{Sk}(D_{L}')$, and $\phi_{d}(\widehat{\psi}_{Sk}(D_{L}'))= \widehat{\psi}_{Sk}(D_{L})$.\par
    Moreover, if two closed braid representatives $B_{L}$, and $B_{L}'$ of $L$, are related by braid conjugation, or braid isotopies (which correspond to allowable Reidemeister  Moves $\Rom{2}$ and $\Rom{3}$ in $A$), then the associated quasi-isomorphisms $\rho_{Sk}:C_{Sk}(B_{L})\to C_{Sk}(B_{L}')$ satisfy $\rho_{Sk}(\widehat{\psi}_{Sk}(B_{L})) = \pm\widehat{\psi}_{Sk}(B_{L}')$.
\end{theorem}

\begin{proof}
We can define $\phi_{s}$ in the following way, see Figure \ref{fig:stab_map}. When the circle corresponding to $P$ is nontrivial, $\phi_{s}$ maps $x\mapsto x\otimes v_{-}$, when $x=v_{-}$. Here the first term of the tensor product denotes labeling on $P$ and the second term denotes labeling on $U$. Similarly, when the circle corresponding to $P$ is trivial, $\phi_{s}(x)= x\otimes v_{-}$, when $x=w_{-}$. We define $\phi_{s}$ to be zero, when $x\in \{v_{+},w_{+}\}$. Note that $\phi_{s}$ is a cochain map. 
\begin{figure}[htp]
    \centering
    \tikzset{every picture/.style={line width=0.75pt}} 

\begin{tikzpicture}[x=0.75pt,y=0.75pt,yscale=-1,xscale=1]

\draw [line width=0.75]    (37.33,63) -- (37.33,123.5) ;
\draw [line width=0.75]    (152.31,84.65) -- (140.62,99.58) ;
\draw [line width=0.75]    (144.15,90.35) -- (139.08,85.27) ;
\draw [line width=0.75]    (153.85,99.73) -- (148.77,94.65) ;
\draw [line width=0.75]    (152.31,84.65) .. controls (156,78.96) and (165.14,85.4) .. (165.08,91.27) .. controls (165.01,97.14) and (157.85,102.5) .. (153.85,99.73) ;
\draw [line width=0.75]    (139.54,62.04) -- (139.08,85.27) ;
\draw [line width=0.75]    (140.62,99.58) -- (140.15,122.81) ;
\draw    (72.5,95.5) -- (114,95.5) ;
\draw [shift={(116,95.5)}, rotate = 180] [color={rgb, 255:red, 0; green, 0; blue, 0 }  ][line width=0.75]    (10.93,-3.29) .. controls (6.95,-1.4) and (3.31,-0.3) .. (0,0) .. controls (3.31,0.3) and (6.95,1.4) .. (10.93,3.29)   ;
\draw [line width=0.75]    (356.67,-4.17) -- (356.67,56.33) ;
\draw    (415.83,27.67) -- (457.33,27.67) ;
\draw [shift={(459.33,27.67)}, rotate = 180] [color={rgb, 255:red, 0; green, 0; blue, 0 }  ][line width=0.75]    (10.93,-3.29) .. controls (6.95,-1.4) and (3.31,-0.3) .. (0,0) .. controls (3.31,0.3) and (6.95,1.4) .. (10.93,3.29)   ;
\draw   (370.94,151.12) .. controls (370.94,145.36) and (375.61,140.7) .. (381.36,140.7) .. controls (387.11,140.7) and (391.78,145.36) .. (391.78,151.12) .. controls (391.78,156.87) and (387.11,161.53) .. (381.36,161.53) .. controls (375.61,161.53) and (370.94,156.87) .. (370.94,151.12) -- cycle ;
\draw [line width=0.75]    (357.33,121.33) -- (357.33,181.83) ;
\draw    (417.17,151.5) -- (458.67,151.5) ;
\draw [shift={(460.67,151.5)}, rotate = 180] [color={rgb, 255:red, 0; green, 0; blue, 0 }  ][line width=0.75]    (10.93,-3.29) .. controls (6.95,-1.4) and (3.31,-0.3) .. (0,0) .. controls (3.31,0.3) and (6.95,1.4) .. (10.93,3.29)   ;
\draw    (490.33,119.67) .. controls (494.67,166.5) and (500.67,121.5) .. (517,142.5) ;
\draw    (490.33,177.5) .. controls (492,134.17) and (500.67,175.67) .. (517,154.67) ;
\draw    (517,142.5) .. controls (518.67,146.5) and (520,149.17) .. (517,154.67) ;
\draw    (356.92,66.92) -- (356.92,105.25) ;
\draw [shift={(356.92,107.25)}, rotate = 270] [color={rgb, 255:red, 0; green, 0; blue, 0 }  ][line width=0.75]    (10.93,-3.29) .. controls (6.95,-1.4) and (3.31,-0.3) .. (0,0) .. controls (3.31,0.3) and (6.95,1.4) .. (10.93,3.29)   ;
\draw    (496.75,64.25) -- (496.75,105.75) ;
\draw [shift={(496.75,107.75)}, rotate = 270] [color={rgb, 255:red, 0; green, 0; blue, 0 }  ][line width=0.75]    (10.93,-3.29) .. controls (6.95,-1.4) and (3.31,-0.3) .. (0,0) .. controls (3.31,0.3) and (6.95,1.4) .. (10.93,3.29)   ;
\draw [color=blue  ,draw opacity=1 ][fill={rgb, 255:red, 155; green, 155; blue, 155 }  ,fill opacity=1 ][line width=0.75]    (357.33,151.12) -- (370.94,151.12) ;
\draw    (284.5,22) -- (284.5,37.75) ;
\draw  [color={rgb, 255:red, 252; green, 3; blue, 3 }  ,draw opacity=1 ][fill={rgb, 255:red, 252; green, 3; blue, 3 }  ,fill opacity=1 ][line width=0.75]  (54,95.32) .. controls (54,93.62) and (52.62,92.25) .. (50.92,92.25) .. controls (49.22,92.25) and (47.85,93.62) .. (47.85,95.32) .. controls (47.85,97.02) and (49.22,98.4) .. (50.92,98.4) .. controls (52.62,98.4) and (54,97.02) .. (54,95.32) -- cycle ;
\draw  [color={rgb, 255:red, 252; green, 3; blue, 3 }  ,draw opacity=1 ][fill={rgb, 255:red, 252; green, 3; blue, 3 }  ,fill opacity=1 ][line width=0.75]  (160,91.72) .. controls (160,90.02) and (158.62,88.65) .. (156.92,88.65) .. controls (155.22,88.65) and (153.85,90.02) .. (153.85,91.72) .. controls (153.85,93.42) and (155.22,94.8) .. (156.92,94.8) .. controls (158.62,94.8) and (160,93.42) .. (160,91.72) -- cycle ;
\draw  [color={rgb, 255:red, 252; green, 3; blue, 3 }  ,draw opacity=1 ][fill={rgb, 255:red, 252; green, 3; blue, 3 }  ,fill opacity=1 ][line width=0.75]  (295.55,29.72) .. controls (295.55,28.02) and (294.18,26.65) .. (292.48,26.65) .. controls (290.78,26.65) and (289.4,28.02) .. (289.4,29.72) .. controls (289.4,31.42) and (290.78,32.8) .. (292.48,32.8) .. controls (294.18,32.8) and (295.55,31.42) .. (295.55,29.72) -- cycle ;
\draw [line width=0.75]    (288.96,147.02) -- (279.41,159.2) ;
\draw [line width=0.75]    (282.3,151.67) -- (278.16,147.52) ;
\draw [line width=0.75]    (290.22,159.33) -- (286.07,155.18) ;
\draw [line width=0.75]    (288.96,147.02) .. controls (291.97,142.37) and (299.44,147.63) .. (299.39,152.42) .. controls (299.34,157.21) and (293.48,161.59) .. (290.22,159.33) ;
\draw [line width=0.75]    (278.02,141.12) -- (278.16,147.52) ;
\draw [line width=0.75]    (279.41,159.2) -- (279.41,166.4) ;
\draw  [color={rgb, 255:red, 252; green, 3; blue, 3 }  ,draw opacity=1 ][fill={rgb, 255:red, 252; green, 3; blue, 3 }  ,fill opacity=1 ][line width=0.75]  (295.44,152.62) .. controls (295.44,150.92) and (294.06,149.54) .. (292.36,149.54) .. controls (290.66,149.54) and (289.28,150.92) .. (289.28,152.62) .. controls (289.28,154.31) and (290.66,155.69) .. (292.36,155.69) .. controls (294.06,155.69) and (295.44,154.31) .. (295.44,152.62) -- cycle ;
\draw  [color={rgb, 255:red, 252; green, 3; blue, 3 }  ,draw opacity=1 ][fill={rgb, 255:red, 252; green, 3; blue, 3 }  ,fill opacity=1 ][line width=0.75]  (372.15,30.32) .. controls (372.15,28.62) and (370.78,27.25) .. (369.08,27.25) .. controls (367.38,27.25) and (366,28.62) .. (366,30.32) .. controls (366,32.02) and (367.38,33.4) .. (369.08,33.4) .. controls (370.78,33.4) and (372.15,32.02) .. (372.15,30.32) -- cycle ;
\draw  [color={rgb, 255:red, 252; green, 3; blue, 3 }  ,draw opacity=1 ][fill={rgb, 255:red, 252; green, 3; blue, 3 }  ,fill opacity=1 ][line width=0.75]  (384.44,151.12) .. controls (384.44,149.42) and (383.06,148.04) .. (381.36,148.04) .. controls (379.66,148.04) and (378.28,149.42) .. (378.28,151.12) .. controls (378.28,152.81) and (379.66,154.19) .. (381.36,154.19) .. controls (383.06,154.19) and (384.44,152.81) .. (384.44,151.12) -- cycle ;
\draw  [color={rgb, 255:red, 252; green, 3; blue, 3 }  ,draw opacity=1 ][fill={rgb, 255:red, 252; green, 3; blue, 3 }  ,fill opacity=1 ][line width=0.75]  (509.69,149.62) .. controls (509.69,147.92) and (508.31,146.54) .. (506.61,146.54) .. controls (504.91,146.54) and (503.53,147.92) .. (503.53,149.62) .. controls (503.53,151.31) and (504.91,152.69) .. (506.61,152.69) .. controls (508.31,152.69) and (509.69,151.31) .. (509.69,149.62) -- cycle ;

\draw (491.33,17.9) node [anchor=north west][inner sep=0.75pt]    {$0$};
\draw (426.67,131.73) node [anchor=north west][inner sep=0.75pt]    {$m$};
\draw (372.5,163.4) node [anchor=north west][inner sep=0.75pt]    {$U$};
\draw (340.5,142.9) node [anchor=north west][inner sep=0.75pt]    {$P$};
\draw (337.5,24.23) node [anchor=north west][inner sep=0.75pt]    {$P$};
\draw (370,72.4) node [anchor=north west][inner sep=0.75pt]    {$\phi _{s}$};
\draw (507.5,68.9) node [anchor=north west][inner sep=0.75pt]    {$\phi _{s}$};
\draw (240.8,19.7) node [anchor=north west][inner sep=0.75pt]    {$C_{Sk}( \ \ \ \ \ \ ) :$};
\draw (32,133.9) node [anchor=north west][inner sep=0.75pt]    {$D_{L}$};
\draw (132,132.9) node [anchor=north west][inner sep=0.75pt]    {$D_{L} '$};
\draw (241.4,141.89) node [anchor=north west][inner sep=0.75pt]    {$C_{Sk}( \ \ \ \ \ \ ) :$};

\end{tikzpicture}
    \caption{Invariance under positive stabilization.}
    \label{fig:stab_map}
\end{figure}
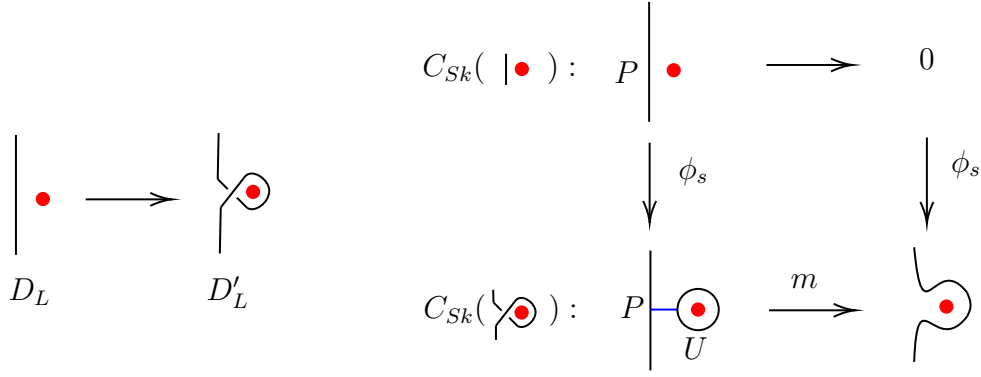\par
On the other hand, we define $\phi_{d}$ in the following way, see Figure \ref{fig:destab_map}. The map $\phi_{d}$ maps $ x\otimes v_{-} \mapsto x$. Here the first term of the tensor product denotes labeling on $P$ and the second term denotes labeling on $U$. We define $\phi_{d}$ to be zero, when labeling on $U$ is $v_{+}$. Now $\phi_{d}: C^{0,sl(L),f_{\text{min}}(D_{L})}_{Sk}(D_{L})\to$ $C^{0,sl(L),f_{\text{min}}(D_{L}')}_{Sk}(D_{L}')$, and $\phi_{d}: C^{0,sl(L),f_{\text{min}}(D_{L}')}_{Sk}(D_{L}') \to C^{0,sl(L),f_{\text{min}}(D_{L})}_{Sk}(D_{L})$ are desired isomorphisms due to Lemma \ref{lemma 2.2}.  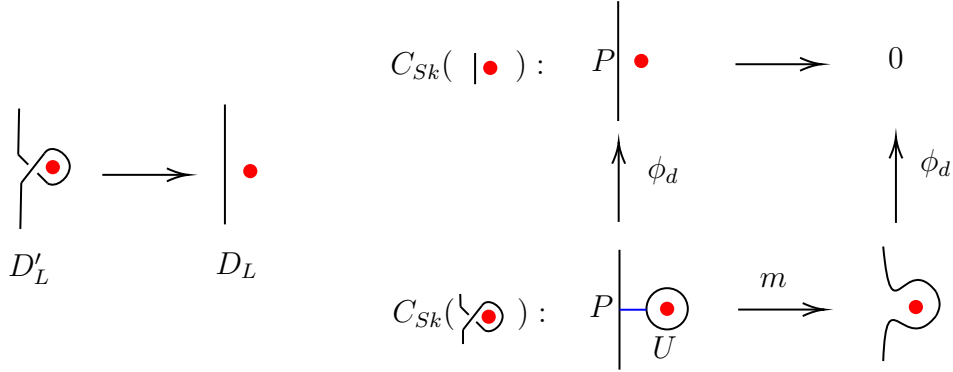
\begin{figure}[htp]
    \centering
    \tikzset{every picture/.style={line width=0.75pt}} 

\begin{tikzpicture}[x=0.75pt,y=0.75pt,yscale=-1,xscale=1]

\draw [line width=0.75]    (160.33,74.5) -- (160.33,135) ;
\draw [line width=0.75]    (69.31,99.15) -- (57.62,114.08) ;
\draw [line width=0.75]    (61.15,104.85) -- (56.08,99.77) ;
\draw [line width=0.75]    (70.85,114.23) -- (65.77,109.15) ;
\draw [line width=0.75]    (69.31,99.15) .. controls (73,93.46) and (82.14,99.9) .. (82.08,105.77) .. controls (82.01,111.64) and (74.85,117) .. (70.85,114.23) ;
\draw [line width=0.75]    (56.54,76.54) -- (56.08,99.77) ;
\draw [line width=0.75]    (57.62,114.08) -- (57.15,137.31) ;
\draw    (98.5,110) -- (140,110) ;
\draw [shift={(142,110)}, rotate = 180] [color={rgb, 255:red, 0; green, 0; blue, 0 }  ][line width=0.75]    (10.93,-3.29) .. controls (6.95,-1.4) and (3.31,-0.3) .. (0,0) .. controls (3.31,0.3) and (6.95,1.4) .. (10.93,3.29)   ;
\draw [line width=0.75]    (358.67,22.33) -- (358.67,82.83) ;
\draw    (417.83,54.17) -- (459.33,54.17) ;
\draw [shift={(461.33,54.17)}, rotate = 180] [color={rgb, 255:red, 0; green, 0; blue, 0 }  ][line width=0.75]    (10.93,-3.29) .. controls (6.95,-1.4) and (3.31,-0.3) .. (0,0) .. controls (3.31,0.3) and (6.95,1.4) .. (10.93,3.29)   ;
\draw   (372.94,177.62) .. controls (372.94,171.86) and (377.61,167.2) .. (383.36,167.2) .. controls (389.11,167.2) and (393.78,171.86) .. (393.78,177.62) .. controls (393.78,183.37) and (389.11,188.03) .. (383.36,188.03) .. controls (377.61,188.03) and (372.94,183.37) .. (372.94,177.62) -- cycle ;
\draw [line width=0.75]    (359.33,147.83) -- (359.33,208.33) ;
\draw    (419.17,178) -- (460.67,178) ;
\draw [shift={(462.67,178)}, rotate = 180] [color={rgb, 255:red, 0; green, 0; blue, 0 }  ][line width=0.75]    (10.93,-3.29) .. controls (6.95,-1.4) and (3.31,-0.3) .. (0,0) .. controls (3.31,0.3) and (6.95,1.4) .. (10.93,3.29)   ;
\draw    (492.33,146.17) .. controls (496.67,193) and (502.67,148) .. (519,169) ;
\draw    (492.33,204) .. controls (494,160.67) and (502.67,202.17) .. (519,181.17) ;
\draw    (519,169) .. controls (520.67,173) and (522,175.67) .. (519,181.17) ;
\draw    (358.92,95.42) -- (358.92,133.75) ;
\draw [shift={(358.92,93.42)}, rotate = 90] [color={rgb, 255:red, 0; green, 0; blue, 0 }  ][line width=0.75]    (10.93,-3.29) .. controls (6.95,-1.4) and (3.31,-0.3) .. (0,0) .. controls (3.31,0.3) and (6.95,1.4) .. (10.93,3.29)   ;
\draw    (498.75,92.75) -- (498.75,134.25) ;
\draw [shift={(498.75,90.75)}, rotate = 90] [color={rgb, 255:red, 0; green, 0; blue, 0 }  ][line width=0.75]    (10.93,-3.29) .. controls (6.95,-1.4) and (3.31,-0.3) .. (0,0) .. controls (3.31,0.3) and (6.95,1.4) .. (10.93,3.29)   ;
\draw [color=blue  ,draw opacity=1 ][line width=0.75]    (359.33,178.08) -- (372.94,178.08) ;
\draw    (286.5,48.5) -- (286.5,64.25) ;
\draw [line width=0.75]    (289.96,176.02) -- (280.41,188.2) ;
\draw [line width=0.75]    (283.3,180.67) -- (279.16,176.52) ;
\draw [line width=0.75]    (291.22,188.33) -- (287.07,184.18) ;
\draw [line width=0.75]    (289.96,176.02) .. controls (292.97,171.37) and (300.44,176.63) .. (300.39,181.42) .. controls (300.34,186.21) and (294.48,190.59) .. (291.22,188.33) ;
\draw [line width=0.75]    (279.02,170.12) -- (279.16,176.52) ;
\draw [line width=0.75]    (280.41,188.2) -- (280.41,195.4) ;
\draw  [color={rgb, 255:red, 252; green, 3; blue, 3 }  ,draw opacity=1 ][fill={rgb, 255:red, 252; green, 3; blue, 3 }  ,fill opacity=1 ][line width=0.75]  (76.6,106.12) .. controls (76.6,104.42) and (75.22,103.05) .. (73.52,103.05) .. controls (71.82,103.05) and (70.45,104.42) .. (70.45,106.12) .. controls (70.45,107.82) and (71.82,109.2) .. (73.52,109.2) .. controls (75.22,109.2) and (76.6,107.82) .. (76.6,106.12) -- cycle ;
\draw  [color={rgb, 255:red, 252; green, 3; blue, 3 }  ,draw opacity=1 ][fill={rgb, 255:red, 252; green, 3; blue, 3 }  ,fill opacity=1 ][line width=0.75]  (176.2,108.12) .. controls (176.2,106.42) and (174.82,105.05) .. (173.12,105.05) .. controls (171.42,105.05) and (170.05,106.42) .. (170.05,108.12) .. controls (170.05,109.82) and (171.42,111.2) .. (173.12,111.2) .. controls (174.82,111.2) and (176.2,109.82) .. (176.2,108.12) -- cycle ;
\draw  [color={rgb, 255:red, 252; green, 3; blue, 3 }  ,draw opacity=1 ][fill={rgb, 255:red, 252; green, 3; blue, 3 }  ,fill opacity=1 ][line width=0.75]  (373.4,52.52) .. controls (373.4,50.82) and (372.02,49.45) .. (370.32,49.45) .. controls (368.62,49.45) and (367.25,50.82) .. (367.25,52.52) .. controls (367.25,54.22) and (368.62,55.6) .. (370.32,55.6) .. controls (372.02,55.6) and (373.4,54.22) .. (373.4,52.52) -- cycle ;
\draw  [color={rgb, 255:red, 252; green, 3; blue, 3 }  ,draw opacity=1 ][fill={rgb, 255:red, 252; green, 3; blue, 3 }  ,fill opacity=1 ][line width=0.75]  (386.44,177.62) .. controls (386.44,175.92) and (385.06,174.54) .. (383.36,174.54) .. controls (381.66,174.54) and (380.28,175.92) .. (380.28,177.62) .. controls (380.28,179.31) and (381.66,180.69) .. (383.36,180.69) .. controls (385.06,180.69) and (386.44,179.31) .. (386.44,177.62) -- cycle ;
\draw  [color={rgb, 255:red, 252; green, 3; blue, 3 }  ,draw opacity=1 ][fill={rgb, 255:red, 252; green, 3; blue, 3 }  ,fill opacity=1 ][line width=0.75]  (511.89,176.72) .. controls (511.89,175.02) and (510.51,173.64) .. (508.81,173.64) .. controls (507.11,173.64) and (505.73,175.02) .. (505.73,176.72) .. controls (505.73,178.41) and (507.11,179.79) .. (508.81,179.79) .. controls (510.51,179.79) and (511.89,178.41) .. (511.89,176.72) -- cycle ;
\draw  [color={rgb, 255:red, 252; green, 3; blue, 3 }  ,draw opacity=1 ][fill={rgb, 255:red, 252; green, 3; blue, 3 }  ,fill opacity=1 ][line width=0.75]  (296.44,181.62) .. controls (296.44,179.92) and (295.06,178.54) .. (293.36,178.54) .. controls (291.66,178.54) and (290.28,179.92) .. (290.28,181.62) .. controls (290.28,183.31) and (291.66,184.69) .. (293.36,184.69) .. controls (295.06,184.69) and (296.44,183.31) .. (296.44,181.62) -- cycle ;
\draw  [color={rgb, 255:red, 252; green, 3; blue, 3 }  ,draw opacity=1 ][fill={rgb, 255:red, 252; green, 3; blue, 3 }  ,fill opacity=1 ][line width=0.75]  (297.24,56.02) .. controls (297.24,54.32) and (295.86,52.94) .. (294.16,52.94) .. controls (292.46,52.94) and (291.08,54.32) .. (291.08,56.02) .. controls (291.08,57.71) and (292.46,59.09) .. (294.16,59.09) .. controls (295.86,59.09) and (297.24,57.71) .. (297.24,56.02) -- cycle ;

\draw (493.33,44.4) node [anchor=north west][inner sep=0.75pt]    {$0$};
\draw (428.67,158.23) node [anchor=north west][inner sep=0.75pt]    {$m$};
\draw (374.5,189.9) node [anchor=north west][inner sep=0.75pt]    {$U$};
\draw (342.5,169.4) node [anchor=north west][inner sep=0.75pt]    {$P$};
\draw (343.5,45.4) node [anchor=north west][inner sep=0.75pt]    {$P$};
\draw (372,98.9) node [anchor=north west][inner sep=0.75pt]    {$\phi _{d}$};
\draw (509.5,95.4) node [anchor=north west][inner sep=0.75pt]    {$\phi _{d}$};
\draw (242.27,45.57) node [anchor=north west][inner sep=0.75pt]    {$C_{Sk}( \ \ \ \ \ \ ) :$};
\draw (243.33,170.93) node [anchor=north west][inner sep=0.75pt]    {$C_{Sk}( \ \ \ \ \ \ ) :$};
\draw (154,147.4) node [anchor=north west][inner sep=0.75pt]    {$D_{L}$};
\draw (49.5,147.9) node [anchor=north west][inner sep=0.75pt]    {$D_{L} '$};

\end{tikzpicture}
    \caption{Invariance under positive destabilization.}
    \label{fig:destab_map}
\end{figure}\par
For the second part, the maps $\rho_{Sk}^{2}$ and $\rho_{Sk}^{3}$, corresponding to the allowable Reidemeister moves $\Rom{2}$ and $\Rom{3}$ in the annulus $A$, can be described using the maps $\rho_{2}$ and $\rho_{3}$ as in \cite[Lemma 1]{Olga2006}, since they preserve the homotopical grading.
\end{proof}

\begin{remark}
    Note that the maps $\phi_{s}$, and $\phi_{d}$ in Proposition \ref{prop 2.2} are not unique. For example, we can define $\phi_{s}$ as $v_{+}\mapsto v_{+}\otimes v_{-}-v_{-}\otimes v_{+}$, and $v_{-}\mapsto v_{-}\otimes v_{-}$, and $w_{+}\mapsto 0$, and $w_{-}\mapsto w_{-}\otimes v_{-}$. Similarly, we can define $\phi_{d}$ differently as: $v_{+}\otimes v_{+}\mapsto 0$, and $v_{-}\otimes v_{+}\mapsto v_{+}$, and $w_{+}\otimes v_{+}\mapsto 0$, and $w_{-}\otimes v_{+}\mapsto 0$, and $x\otimes v_{-}\mapsto x$, where $x$ is any labeling on $P$.
\end{remark}

\section{Flow Category for Skein Homology and Invariance of Khovanov skein spectrum}\label{section 3}

In this section, we will first describe a framed flow category for the link diagram $D_L\subset A$ and then feed this into the Cohen-Jones-Segal construction, see \cite{CohJonSeg1995} and \cite{LipSar2014} to obtain a suspension spectrum. We will then prove the invariance of this suspension spectrum.

\subsection{Resolution Moduli spaces}
\begin{definition}[\cite{LipSar2014}, Definition 4.4]
    Let $u,v\in \{0,1\}^n$ be two states such that $v\prec u$, and $|u|-|v|=m$. Let $j_{1}<\cdots< j_{m}$ be the $m$ indices where $u$ and $v$ differ. Corresponding to this pair of states $u,v$, Lipshitz and Sarkar \cite{LipSar2014} defined an inclusion functor 
    $$ 
        \mathcal{I}_{u,v}: \scrC_{C}(m) \hookrightarrow \scrC_{C}(n) 
    $$
    in the following way. Given an object $w \in \{0,1\}^m$ in $\scrC_{C}(m)$, we define an object $w'\in \{0,1\}^n$ in $\scrC_{C}(n)$ by setting $w'_{i}= 0$ if $u_{i}=0$, and $w'_{i}= 1$ if $v_{i}=1$, and for the rest $w'_{j_{i}} = w_{i}$. The full subcategory spanned by the objects $\big\{w'\,|\, w\in \text{Ob}\big(\scrC_{C}(m)\big)\big\}$ is isomorphic to $\scrC_{C}(m)$, and $\mathcal{I}_{u,v}$ is defined to be this isomorphism.
\end{definition}
We will first associate to each index $n$ basic decorated resolution configuration $(D,x,y)$ in the annulus an $(n-1)$-dimensional $<n-1>$-manifold $\calM(D,x,y)$ together with an $(n-1)$-map
$$\calF:\calM(D,x,y)\longrightarrow \calM_{\scrC_C(n)}(\overline{1},\overline{0})$$
These spaces and maps will be constructed inductively using \cite[Proposition 5.2]{LipSar2014}, satisfying the following properties:
\begin{enumerate}[label = (RM-\arabic*)]
    \item Let $(D,x,y)$ be an index $n$ basic decorated resolution configuration in the annulus and $(E,z)$ be an index $m$ labeled resolution configuration in the annulus such that $(E,z) \in P(D,x,y)$. Let $x|$ and $y|$ the induced labelings on $s(E\setminus s(D)) = s(D)\setminus E$ and $D\setminus E$ respectively, and $z|$ the induced labelings for both $s(D\setminus E) = E\setminus D$ and $E\setminus s(D)$. Then there is a composition map: $$\circ : \calM(D\setminus E,z|,y|)\times \calM(E\setminus s(D),x|,z|)\longrightarrow \calM(D,x,y),$$
    such that the following diagram commutes:
    \[\begin{tikzcd}
	\calM(D\setminus E,z|,y|)\times \calM(E\setminus s(D),x|,z|) & \calM(D,x,y) \\
	\calM_{\scrC_C(n-m)}(\overline{1},\overline{0})\times \calM_{\scrC_C(m)}(\overline{1},\overline{0}) \\
	\calM_{\scrC_C(n)}(v,\overline{0})\times \calM_{\scrC_C(n)}(\overline{1},v) & \calM_{\scrC_C(n)}(\overline{1},\overline{0}).
	\arrow["\circ", from=1-1, to=1-2]
	\arrow["{\calF \times \calF}"', from=1-1, to=2-1]
	\arrow["\calF", from=1-2, to=3-2]
	\arrow["{ \mathcal{I}_{v,\overline{0}} \times \mathcal{I}_{\overline{1},v} }"', from=2-1, to=3-1]
	\arrow["\circ"', from=3-1, to=3-2]
\end{tikzcd}\]
Here $v$ is the state $(v_{1},\dots, v_{n})\in \{0,1\}^n$ such that $v_{i}=0$ if and only if the $i$-th arc of $D$ is an arc of $E$ as well.
    \item The faces of $\calM(D,x,y)$ are given by
    \begin{align*}
        \partial_{exp,i}\calM(D,x,y) \coloneq \coprod_{\substack{(E,z)\in P(D,x,y)\\ \text{ind}(D\setminus E)=i}} \circ (\calM(D\setminus E,z|,y|)\times \calM(E\setminus s(D),x|,z|)).
    \end{align*}
    \item The map $\calF$ is a covering map, and a local diffeomorphism.
    \item The covering map $\calF$ is trivial on each component of $\calM(D,x,y)$.
\end{enumerate}

Let $L \subset A\times I$ be an oriented link diagram, with $n$-crossings and an ordering of the crossings as before. We will construct Khovanov skein flow category, $\scrC_{Sk}(D_{L})$ similar to Khovanov flow category \cite[Definition 5.3]{LipSar2014} as a cover of the cube flow category $\scrC_C(n(D_{L}))$. We refer to \cite[Definition 3.28]{LipSar2014} for the notion of cover of a flow category. 

\subsection{$0$-dimensional moduli spaces}
For an index $1$ basic decorated resolution configuration $(D,x,y)$ in the annulus, we define $\calM(D,x,y)$ to consist of a single point and $$\calF: \calM(D,x,y)\longrightarrow \calM_{\scrC_C(1)}(\overline{1},\overline{0})$$ is the obvious map sending a point to a point.
\subsection{$1$-dimensional moduli spaces}

\begin{definition}
An index $2$ basic resolution configuration $D$ in the annulus is said to be a \textit{ladybug configuration} if the following conditions are satisfied.
\begin{enumerate}
    \item $Z(D)$ consists of a single trivial circle, say $Z$;
    \item The endpoints of the two arcs in $R(D)$, say $R_1$ and $R_2$ which intersects the circle $Z$ alternates around it.
\end{enumerate}
 We have exactly two ladybug configurations, depending on the choice of the axis which are given in Figure \ref{fig:ladybug}. We will mostly discuss case (ii) configuration as the other one is the same one in the Khovanov case.
\begin{figure}[htp]
    \centering
    \vspace{-0.8cm}

\tikzset{every picture/.style={line width=0.75pt}} 

\begin{tikzpicture}[x=0.75pt,y=0.75pt,yscale=-1,xscale=1]

\draw   (41,82.55) .. controls (41,65.98) and (54.43,52.55) .. (71,52.55) .. controls (87.57,52.55) and (101,65.98) .. (101,82.55) .. controls (101,99.12) and (87.57,112.55) .. (71,112.55) .. controls (54.43,112.55) and (41,99.12) .. (41,82.55) -- cycle ;
\draw [color=blue  ,draw opacity=1 ]   (71,52.55) -- (71,112.55) ;
\draw [color=blue  ,draw opacity=1 ]   (43,72.55) .. controls (-1,4.55) and (139,1.55) .. (100,72.55) ;
\draw  [color={rgb, 255:red, 252; green, 3; blue, 3 }  ,draw opacity=1 ][fill={rgb, 255:red, 252; green, 3; blue, 3 }  ,fill opacity=1 ][line width=0.75]  (116,79.92) .. controls (116,78.22) and (114.62,76.85) .. (112.92,76.85) .. controls (111.22,76.85) and (109.85,78.22) .. (109.85,79.92) .. controls (109.85,81.62) and (111.22,83) .. (112.92,83) .. controls (114.62,83) and (116,81.62) .. (116,79.92) -- cycle ;
\draw   (200,81.55) .. controls (200,64.98) and (213.43,51.55) .. (230,51.55) .. controls (246.57,51.55) and (260,64.98) .. (260,81.55) .. controls (260,98.12) and (246.57,111.55) .. (230,111.55) .. controls (213.43,111.55) and (200,98.12) .. (200,81.55) -- cycle ;
\draw [color=blue  ,draw opacity=1 ]   (230,51.55) -- (230,111.55) ;
\draw [color=blue  ,draw opacity=1 ]   (202,71.55) .. controls (158,3.55) and (298,0.55) .. (259,71.55) ;
\draw  [color={rgb, 255:red, 252; green, 3; blue, 3 }  ,draw opacity=1 ][fill={rgb, 255:red, 252; green, 3; blue, 3 }  ,fill opacity=1 ][line width=0.75]  (233,37.92) .. controls (233,36.22) and (231.62,34.85) .. (229.92,34.85) .. controls (228.22,34.85) and (226.85,36.22) .. (226.85,37.92) .. controls (226.85,39.62) and (228.22,41) .. (229.92,41) .. controls (231.62,41) and (233,39.62) .. (233,37.92) -- cycle ;

\draw (62,120) node [anchor=north west][inner sep=0.75pt]  [font=\footnotesize] [align=left] {(i)};
\draw (220,120) node [anchor=north west][inner sep=0.75pt]  [font=\footnotesize] [align=left] {(ii)};
\draw (100,19) node [anchor=north west][inner sep=0.75pt]  [font=\footnotesize] [align=left] {1};
\draw (260,19) node [anchor=north west][inner sep=0.75pt]  [font=\footnotesize] [align=left] {1};
\draw (74.5,76.5) node [anchor=north west][inner sep=0.75pt]  [font=\footnotesize] [align=left] {2};
\draw (234.5,76.5) node [anchor=north west][inner sep=0.75pt]  [font=\footnotesize] [align=left] {2};

\end{tikzpicture}
    \caption{Ladybug configurations.}
    \label{fig:ladybug}
\end{figure}
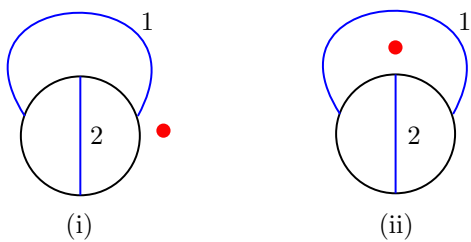
\end{definition}

An index $2$ basic decorated resolution configuration $(D,x,y)$ in the annulus is said to be a \textit{ladybug configuration} if $D$ is a ladybug configuration. The restrictions on the homotopical gradings and the triviality of the circle force such configurations to have, $y(Z) = w_{+}$ and $x(Z^s) = w_{-}$ where $Z^s$ is the unique circle in $Z(s(D))$. Figure \ref{fig:ladybug_decorated_resolution} illustrates the poset $P(D,w_-,w_+)$ for the ladybug configuration (ii) in Figure \ref{fig:ladybug}.

For any index $2$ basic decorated resolution configuration $(D,x,y)$ in the annulus, the number of labeled resolution configurations between $(D,y)$ and $(s(D),x)$, denoted by $m$ is $2$ whenever $D$ has a leaf or a co-leaf and is $4$ when $(D,x,y)$ is a ladybug configuration which is the only index $2$ basic configuration without leaves or co-leaves.

\begin{remark}
Notice that in the definition of the ladybug configuration, we require the circle to be trivial. There is a basic resolution configuration in the annulus where the circle is non-trivial as shown in Figure \ref{fig:non_ladybug} below.
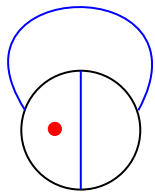
\begin{figure}[htp]
    \centering
    \vspace{-0.5cm}

\tikzset{every picture/.style={line width=0.75pt}} 

\begin{tikzpicture}[x=0.75pt,y=0.75pt,yscale=-1,xscale=1]

\draw   (26,79.55) .. controls (26,62.98) and (39.43,49.55) .. (56,49.55) .. controls (72.57,49.55) and (86,62.98) .. (86,79.55) .. controls (86,96.12) and (72.57,109.55) .. (56,109.55) .. controls (39.43,109.55) and (26,96.12) .. (26,79.55) -- cycle ;
\draw [color=blue  ,draw opacity=1 ]   (56,49.55) -- (56,109.55) ;
\draw [color=blue  ,draw opacity=1 ]   (28,69.55) .. controls (-16,1.55) and (124,-1.45) .. (85,69.55) ;
\draw  [color={rgb, 255:red, 252; green, 3; blue, 3 }  ,draw opacity=1 ][fill={rgb, 255:red, 252; green, 3; blue, 3 }  ,fill opacity=1 ][line width=0.75]  (46,78.92) .. controls (46,77.22) and (44.62,75.85) .. (42.92,75.85) .. controls (41.22,75.85) and (39.85,77.22) .. (39.85,78.92) .. controls (39.85,80.62) and (41.22,82) .. (42.92,82) .. controls (44.62,82) and (46,80.62) .. (46,78.92) -- cycle ;

\end{tikzpicture}
    \caption{A resolution configuration with a non-trivial circle but this does not give a decorated resolution configuration.}
    \label{fig:non_ladybug}
\end{figure}   
However, we choose to omit this case from our definition as it does not yield any basic decorated resolution configuration. For any choice of label on the circle $Z$, the homotopical gradings would force the label on $ Z^s$ to be $0$.
\end{remark}

\begin{figure}[htp]
    \centering
    \input{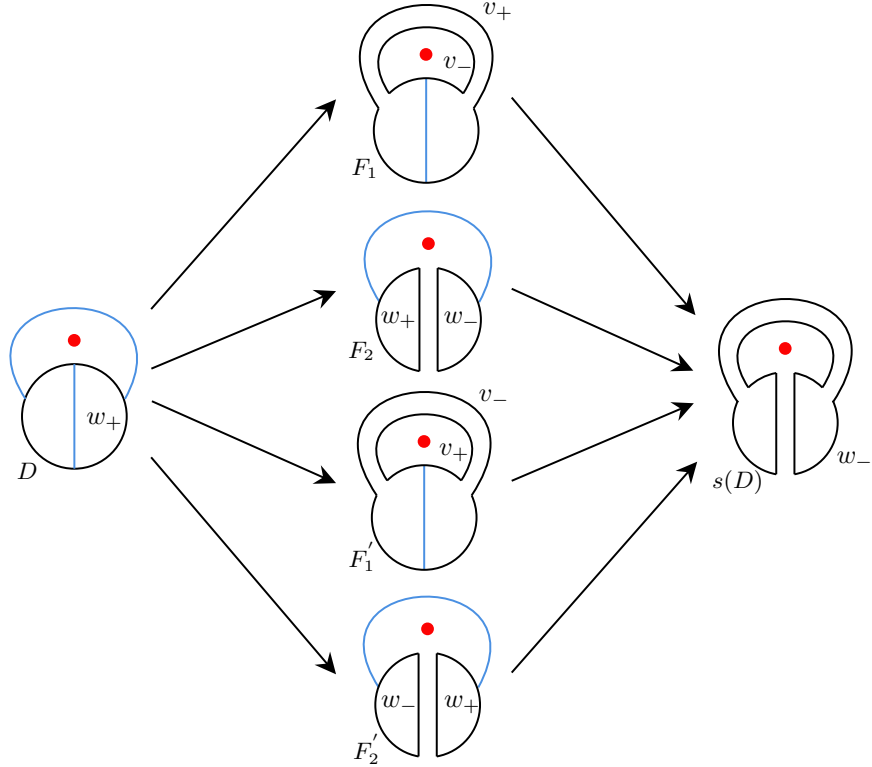}
    \caption{The poset $P(D,w_-,w_+)$ for the ladybug configuration (ii).}
    \label{fig:ladybug_decorated_resolution}
\end{figure}

For $m=2$, the boundary $\partial_{exp}\calM(D,x,y)$ has $2$ points and by the construction from \cite[Proposition 5.2]{LipSar2014} $\calM(D,x,y)$ is an interval and the map $\calF$ is a diffeomorphism. 

For $m=4$, there is a choice involved; we need to decide which pairs of the $4$ points in $\partial_{exp}\calM(D,x,y)$ bounds intervals. We will only be considering the case (ii) ladybug configuration here. Let $Z$ denote the unique circle in $Z(D)$, and $R(D) = \{R_1,R_2\}$ then for each $i\in \{1,2\}$ surgery along $R_i$ consists of two circles, $Z_{i1}$ and $Z_{i2}$, see Figure \ref{fig:ladybug_resolution}. 

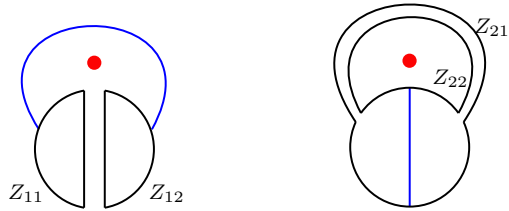
\begin{figure}[htp]
    \centering
    \tikzset{every picture/.style={line width=0.75pt}} 

\begin{tikzpicture}[x=0.75pt,y=0.75pt,yscale=-1,xscale=1]

\draw [color=blue  ,draw opacity=1 ]   (43,72.55) .. controls (-1,4.55) and (139,1.55) .. (100,72.55) ;
\draw  [color={rgb, 255:red, 252; green, 3; blue, 3 }  ,draw opacity=1 ][fill={rgb, 255:red, 252; green, 3; blue, 3 }  ,fill opacity=1 ][line width=0.75]  (74,38.92) .. controls (74,37.22) and (72.62,35.85) .. (70.92,35.85) .. controls (69.22,35.85) and (67.85,37.22) .. (67.85,38.92) .. controls (67.85,40.62) and (69.22,42) .. (70.92,42) .. controls (72.62,42) and (74,40.62) .. (74,38.92) -- cycle ;
\draw [color=blue  ,draw opacity=1 ]   (230,51.55) -- (230,111.55) ;
\draw [color={rgb, 255:red, 0; green, 0; blue, 0 }  ,draw opacity=1 ]   (202.8,68.87) .. controls (150,-9.73) and (310,-10.73) .. (257.2,68.87) ;
\draw  [color={rgb, 255:red, 252; green, 3; blue, 3 }  ,draw opacity=1 ][fill={rgb, 255:red, 252; green, 3; blue, 3 }  ,fill opacity=1 ][line width=0.75]  (233,37.92) .. controls (233,36.22) and (231.62,34.85) .. (229.92,34.85) .. controls (228.22,34.85) and (226.85,36.22) .. (226.85,37.92) .. controls (226.85,39.62) and (228.22,41) .. (229.92,41) .. controls (231.62,41) and (233,39.62) .. (233,37.92) -- cycle ;
\draw  [draw opacity=0] (65.79,112.1) .. controls (51.7,109.63) and (41,97.34) .. (41,82.55) .. controls (41,67.76) and (51.7,55.47) .. (65.79,53) -- (71,82.55) -- cycle ; \draw   (65.79,112.1) .. controls (51.7,109.63) and (41,97.34) .. (41,82.55) .. controls (41,67.76) and (51.7,55.47) .. (65.79,53) ;  
\draw  [draw opacity=0] (76.21,53) .. controls (90.3,55.47) and (101,67.76) .. (101,82.55) .. controls (101,97.34) and (90.3,109.63) .. (76.21,112.1) -- (71,82.55) -- cycle ; \draw   (76.21,53) .. controls (90.3,55.47) and (101,67.76) .. (101,82.55) .. controls (101,97.34) and (90.3,109.63) .. (76.21,112.1) ;  
\draw    (65.79,53) -- (65.79,112.1) ;
\draw    (76.21,53) -- (76.21,112.1) ;
\draw  [draw opacity=0] (257.2,68.87) .. controls (259,72.72) and (260,77.02) .. (260,81.55) .. controls (260,98.12) and (246.57,111.55) .. (230,111.55) .. controls (213.43,111.55) and (200,98.12) .. (200,81.55) .. controls (200,77.02) and (201,72.72) .. (202.8,68.87) -- (230,81.55) -- cycle ; \draw   (257.2,68.87) .. controls (259,72.72) and (260,77.02) .. (260,81.55) .. controls (260,98.12) and (246.57,111.55) .. (230,111.55) .. controls (213.43,111.55) and (200,98.12) .. (200,81.55) .. controls (200,77.02) and (201,72.72) .. (202.8,68.87) ;  
\draw  [draw opacity=0] (205.42,64.34) .. controls (210.85,56.61) and (219.83,51.55) .. (230,51.55) .. controls (240.17,51.55) and (249.15,56.61) .. (254.58,64.34) -- (230,81.55) -- cycle ; \draw   (205.42,64.34) .. controls (210.85,56.61) and (219.83,51.55) .. (230,51.55) .. controls (240.17,51.55) and (249.15,56.61) .. (254.58,64.34) ;  
\draw [color={rgb, 255:red, 0; green, 0; blue, 0 }  ,draw opacity=1 ]   (205.42,64.34) .. controls (171,0.27) and (291,-0.73) .. (254.58,64.34) ;

\draw (26,99) node [anchor=north west][inner sep=0.75pt]  [font=\tiny] [align=left] {$\displaystyle Z_{11}$};
\draw (97,99) node [anchor=north west][inner sep=0.75pt]  [font=\tiny] [align=left] {$\displaystyle Z_{12}$};
\draw (261,15) node [anchor=north west][inner sep=0.75pt]  [font=\tiny] [align=left] {$\displaystyle Z_{21}$};
\draw (240,41) node [anchor=north west][inner sep=0.75pt]  [font=\tiny] [align=left] {$\displaystyle Z_{22}$};

\end{tikzpicture}
    \caption{The two different resolutions of the ladybug configuration (ii).}
    \label{fig:ladybug_resolution}
\end{figure}

The \textit{right pair} in \cite[Section 5.4]{LipSar2014} then determines a bijection called the \textit{ladybug matching} between $\{Z_{11},Z_{12}\}$ and $\{Z_{21},Z_{22}\}$ by

\begin{align*}
    Z_{11} \longleftrightarrow Z_{21} \text{ and } Z_{12} \longleftrightarrow Z_{22}. 
\end{align*}

The four points in $\partial_{exp}\calM(D,x,y)$ are given by
\begin{align*}
    a &= [(D,w_{+})\prec' (s_{R_1}(D),v_{-}v_{+})\prec' (s(D),w_{-})]\\
    b &= [(D,w_{+})\prec' (s_{R_1}(D),v_{+}v_{-})\prec' (s(D),w_{-})]\\
    c &= [(D,w_{+})\prec' (s_{R_2}(D),w_{-}w_{+})\prec' (s(D),w_{-})]\\
    d &= [(D,w_{+})\prec' (s_{R_2}(D),w_{+}w_{-})\prec' (s(D),w_{-})]   
\end{align*}
So, using the ladybug matching we can define $\calM(D,x,y)$ to consist of two intervals, with boundaries $a\sqcup c$ and $b\sqcup d$ and the covering map $\calF:\calM(D,x,y)\longrightarrow \calM_{\scrC_C(2)}(\overline{1},\overline{0})$ in a unique way compatible with our choices, sending $a$ and $b$ to one endpoint of $\calM_{\scrC_C(2)}(\overline{1},\overline{0})$ and $c$ and $d$ to the other.
\subsection{$2$-dimensional moduli spaces}
For any index $3$ basic decorated resolution configuration in the annulus, we will be using \cite[Proposition 5.2]{LipSar2014} to define the $2$-dimensional moduli spaces $\calM(D,x,y)$. From the constructed $0$ - and $1$ - dimensional moduli spaces and the defined $\calF$ on them, we get a covering map,
$$\calF_{\partial}:\partial_{exp}\calM(D,x,y)\longrightarrow \partial \calM_{\scrC_C(3)}(\overline{1},\overline{0}).$$
Then by (E-2) of \cite[Proposition 5.2]{LipSar2014} $\calF_{\partial}$ is a covering map and in order to construct $\calM(D,x,y)$ by (E-3) of Proposition 5.2 in \cite{LipSar2014}, it is enough to show that the map is a trivial covering map on each component. Since, $\partial \calM_{\scrC_C(3)}(\overline{1},\overline{0})$ being the boundary of a permutohedron of order $3$, is a $6$-cycle it suffices to show that $\partial_{exp}\calM(D,x,y)$ is a disjoint union of $6$-cycles.

\begin{lemma}
    The graphs $\partial_{exp}\calM (D,x,y)$ and $\partial_{exp}\calM_*(D^*,y^*,x^*)$ are isomorphic.
\end{lemma}

\begin{proof}
For any index $n$ decorated resolution configuration $(E,u,v)$ in the annulus, there is a natural isomorphism, $f_E$ from $P(E,u,v)$ to the reverse poset of $P(E^*,v^*,u^*)$. Let $(F,w)\in P(E,u,v)$ be a labeled resolution configuration in the annulus, such that $F = s_{T}(E)$ for $ \emptyset \subsetneq T \subsetneq R(E)$. Then, $f_E((F,w)) = (f_E(F),w^*)$ is given by, $f_E(F) = s_{T^*}(E^*)$ where $T^* = \{R^*_{n-i+1} : R_i\in R(E)\setminus T\}$, and $w$ and $w^*$ are labelings which are dual to each other. Also $f_{E}\big((s(E),u)\big)= \big(E^{*}, u^{*}\big)$, and $f_{E}\big((E,v)\big)= \big(s(E^{*}), v^{*}\big)$.\par
For any index $2$ decorated resolution configuration $(E,u,v)$ in the annulus, it can be seen that the vertices $[(E,v)\prec'(F_1,w_1)\prec'(s(E),u)]$ and $[(E,v)\prec'(F_2,w_2)\prec'(s(E),u)]$ bound an interval in $\calM(E,u,v)$ if and only if the vertices $[(f_E(s(E)),u^*)\prec'(f_E(F_1),w_1^*)\prec'(f_E(E),v^*)]$ and $[(f_E(s(E)),u^*)\prec'(f_E(F_2),w_2^*)\prec'(f_E(E),v^*)]$ together bound an interval in $\calM_{*}(E^*,v^*,u^*)$; see Figure \ref{fig:ladybug_decorated_resolution} and \ref{fig:dual_ladybug_decorated_resolution}.
Note that $\calM_{*}$ denotes opposite matching of the ladybug matching. \par
Here, in graphs, $\partial_{exp}\calM (D,x,y)$ and $\partial_{exp}\calM_*(D^*,y^*,x^*)$, the vertices correspond to the maximal chains in $P(D,x,y)$ and $P(D^*,y^*,x^*)$, respectively, and the edges correspond to bounding an interval as discussed above in the index $2$ resolution configurations. Therefore, $f_D$ defined above induces a bijection between the vertices and the corresponding edges. See the Figure \ref{fig:dual_ladybug_decorated_resolution} below for the explicit map $f_D$ when $D$ is a ladybug configuration.

\end{proof}

\begin{center}
\begin{figure}[htp]
    \centering
    \input{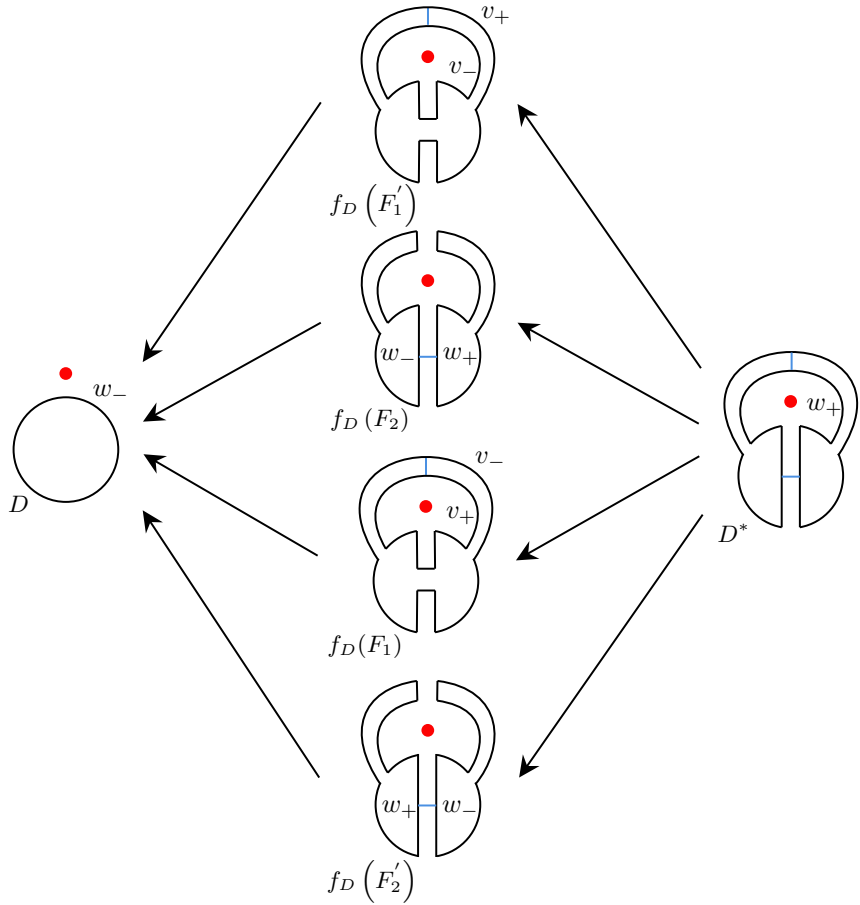}
    \caption{The poset $P(D^*,w_-,w_+)$ for the ladybug configuration (ii) and the bijection $f_D$ from $P(D,w_-,w_+)$ to $P(D^*,w_-,w_+)$.}
    \label{fig:dual_ladybug_decorated_resolution}
\end{figure}
\end{center}
\vspace{-2cm}
\begin{lemma}
    Suppose $D$ has a leaf. Then $\partial_{exp}\calM(D,x,y)$ is either a $6$-cycle or a disjoint union of two $6$-cycles.
\end{lemma}

\begin{proof}
Let $(D,x,y)$ be an index $3$ basic decorated resolution configuration with $Z\in Z(D)$ the leaf and $R\in R(D)$ the surgery arc with one boundary component on $Z$. By \cite[Lemma 2.14]{LipSar2014}, there is a naturally associated index $2$ decorated configuration $(D',x',y')$ such that $P(D,x,y) = P(D',x',y')\times \{0,1\}$. 

If $D'$ is not a ladybug configuration then $P(D,x,y) = \{0,1\}^3$ and $\partial_{exp}\calM(D,x,y)$ is a $6$-cycle.

Now suppose $D'$ is a ladybug configuration, then up to isotopy in the annulus and vertical reflection across the page, Figure \ref{fig:ladybug_leaf_config} gives all the possible index 3 configurations with a leaf and a ladybug. Since $D'$ is a ladybug configuration, the circle is always labeled $w_+$, regardless of where the leaf is placed or attached.

Let the four maximal chains in $P(D',x',y')$ are $e_{i}=[b\prec' c_{i}\prec' d]$ for $1\leq i \leq 4$ and the ladybug matching be $e_{1}\xleftrightarrow{} e_{2}$ and $e_{3}\xleftrightarrow{} e_{4}$. Then, it can be seen that the twelve vertices in $\partial_{exp}\mathcal{M}(D,x,y)$ i.e. the maximal chains in $P(D,x,y)$ are the following.
    \begin{align*}
        u_{i} &= \big[(b,0)\prec' (c_{i},0) \prec' (d,0) \prec' (d,1)\big],\\
        v_{i} &= \big[(b,0)\prec' (c_{i},0) \prec' (c_{i},1) \prec' (d,1)\big],\\
        w_{i} &= \big[(b,0)\prec' (b,1) \prec' (c_{i},1) \prec' (d,1)\big]
    \end{align*}
for $1\leq i \leq 4$, see Figure \ref{fig:ladybug_leaf}.
It can be seen that the following edges exist independently of our choice of ladybug matching:
\begin{center}
\begin{tabular}{c c c c}
    $v_1-u_1$ & $v_1-w_1$ & $v_2-u_2$ & $v_2-w_2$\\
    $v_3-u_3$ & $v_3-w_3$ & $v_4-u_4$ & $v_4-w_4$ 
\end{tabular}   
\end{center}

\begin{center}
\begin{figure}[htp]
    \centering
    \input{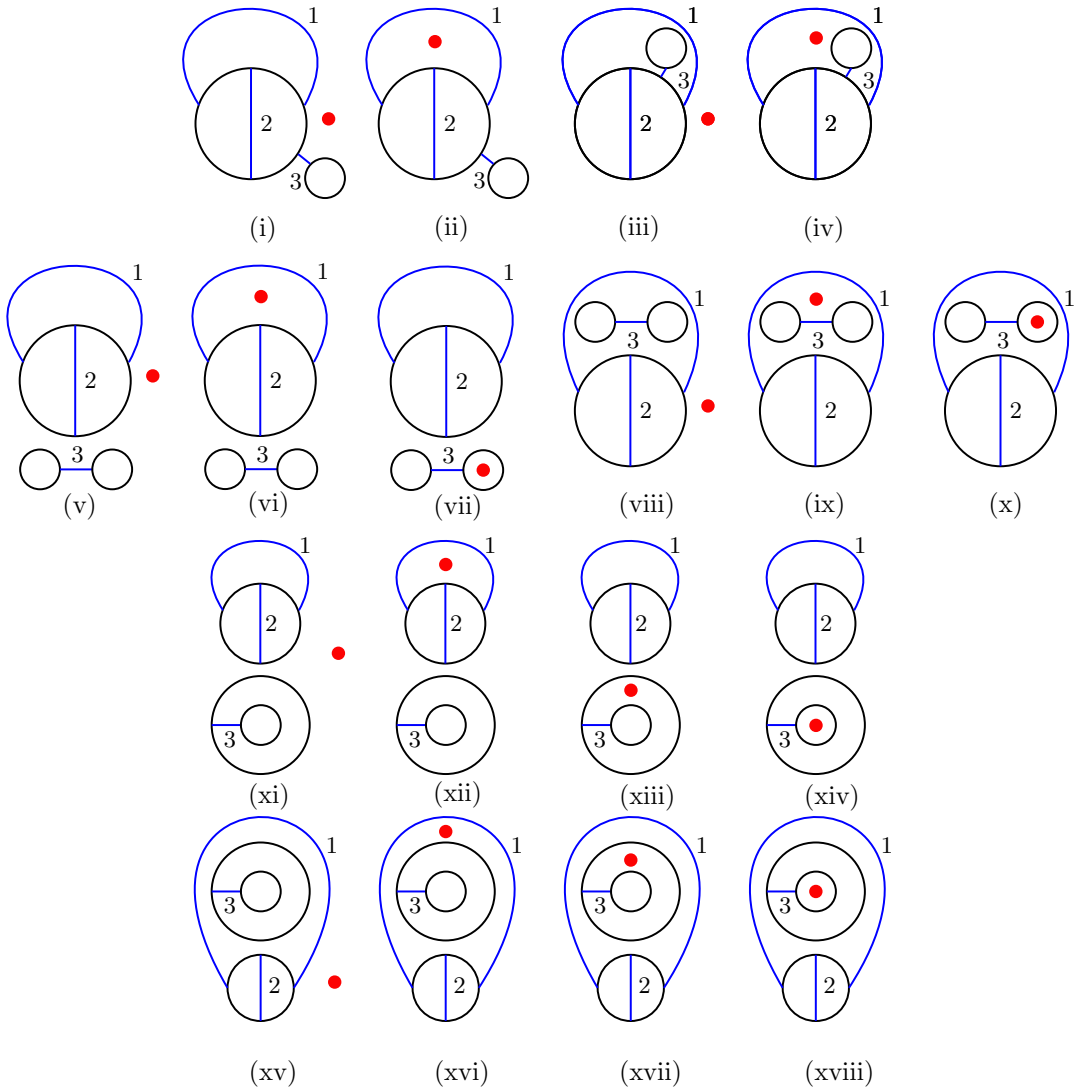}
    \caption{Up to isotopy in $A$ and reordering of the arcs, these are the only basic index 3 resolution configurations with leaves and ladybugs.}
    \label{fig:ladybug_leaf_config}
\end{figure}
\end{center}

\begin{center}
\begin{figure}[htp]
    \centering
    \input{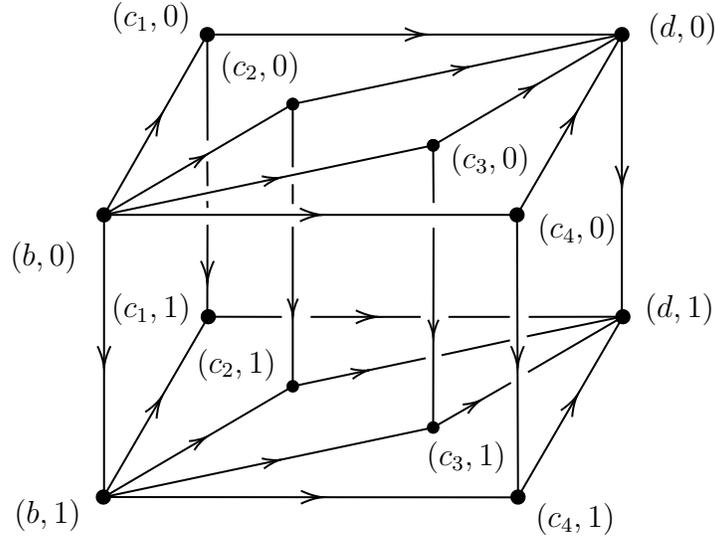}
    \caption{The poset $P(D,x,y) = P(D',x',y')\times \{0,1\}$ for $D$ a configuration with both leaf and ladybug.}
    \label{fig:ladybug_leaf}
\end{figure}
\end{center}

  The rest of the arguments follow from \cite[Lemma 5.14]{LipSar2014}, it suffices to verify for all configurations. For configurations (i) - (iv) in Figure \ref{fig:ladybug_leaf_config} one of the boundary components of the surgery arc, $R$ lies on the circle in the ladybug configuration. Since, $R$ is an $m$-arc, the label on $Z$ is always $w_+$. Consequently, if $e_{i}\xleftrightarrow{} e_{j}$ is the ladybug matching then, this gives new edges joining $u_i$ to $u_j$ and $w_i$ to $w_j$.
 For the rest of the configurations in Figure \ref{fig:ladybug_leaf_config}, since $R$ is completely disjoint from $D'$, the same result follows. Therefore, the following edges come from the ladybug matching:
 \begin{center}
\begin{tabular}{c c c c}
    $u_1-u_2$ & $u_3-u_4$ & $w_1-w_2$ & $w_3-w_4$
\end{tabular}   
\end{center}
 So, $\partial_{exp}\calM(D,x,y)$ is a disjoint union of two $6$-cycles and the components are 
\begin{center}
\begin{tabular}{c c c}
    $v_1-u_1-u_2-v_2-w_2-w_1-v_1$ & \text{and} & $v_3-u_3-u_4-v_4-w_4-w_3-v_3.$
\end{tabular}   
\end{center}
\end{proof}

\begin{remark}
Notice that there are more resolution configurations with both the leaf and ladybug, as shown in Figure \ref{fig:non_ladybug_leaf_config}. However, we choose not to consider these as any choice of label would not result in an index 3 decorated resolution configuration. The homotopical gradings would force the label after the surgery to be $0$.
\begin{center}
\begin{figure}[htp]
    \centering
    \vspace{-0.5cm}

\tikzset{every picture/.style={line width=0.75pt}} 

\begin{tikzpicture}[x=0.75pt,y=0.75pt,yscale=-1,xscale=1]

\draw   (263,127.55) .. controls (263,110.98) and (276.43,97.55) .. (293,97.55) .. controls (309.57,97.55) and (323,110.98) .. (323,127.55) .. controls (323,144.12) and (309.57,157.55) .. (293,157.55) .. controls (276.43,157.55) and (263,144.12) .. (263,127.55) -- cycle ;
\draw [color=blue  ,draw opacity=1 ]   (293,97.55) -- (293,157.55) ;
\draw [color=blue  ,draw opacity=1 ]   (265,117.55) .. controls (221,49.55) and (361,46.55) .. (322,117.55) ;
\draw  [color={rgb, 255:red, 252; green, 3; blue, 3 }  ,draw opacity=1 ][fill={rgb, 255:red, 252; green, 3; blue, 3 }  ,fill opacity=1 ][line width=0.75]  (336.08,157) .. controls (336.08,155.3) and (334.7,153.92) .. (333,153.92) .. controls (331.3,153.92) and (329.92,155.3) .. (329.92,157) .. controls (329.92,158.7) and (331.3,160.08) .. (333,160.08) .. controls (334.7,160.08) and (336.08,158.7) .. (336.08,157) -- cycle ;
\draw   (322.25,157) .. controls (322.25,151.06) and (327.06,146.25) .. (333,146.25) .. controls (338.94,146.25) and (343.75,151.06) .. (343.75,157) .. controls (343.75,162.94) and (338.94,167.75) .. (333,167.75) .. controls (327.06,167.75) and (322.25,162.94) .. (322.25,157) -- cycle ;
\draw [color=blue  ,draw opacity=1 ]   (318.5,144) -- (325,149.25) ;
\draw   (370,127.55) .. controls (370,110.98) and (383.43,97.55) .. (400,97.55) .. controls (416.57,97.55) and (430,110.98) .. (430,127.55) .. controls (430,144.12) and (416.57,157.55) .. (400,157.55) .. controls (383.43,157.55) and (370,144.12) .. (370,127.55) -- cycle ;
\draw [color=blue  ,draw opacity=1 ]   (400,97.55) -- (400,157.55) ;
\draw [color=blue  ,draw opacity=1 ]   (372,117.55) .. controls (328,49.55) and (468,46.55) .. (429,117.55) ;
\draw   (370,127.55) .. controls (370,110.98) and (383.43,97.55) .. (400,97.55) .. controls (416.57,97.55) and (430,110.98) .. (430,127.55) .. controls (430,144.12) and (416.57,157.55) .. (400,157.55) .. controls (383.43,157.55) and (370,144.12) .. (370,127.55) -- cycle ;
\draw [color=blue  ,draw opacity=1 ]   (400,97.55) -- (400,157.55) ;
\draw [color=blue  ,draw opacity=1 ]   (372,117.55) .. controls (328,49.55) and (468,46.55) .. (429,117.55) ;
\draw  [color={rgb, 255:red, 252; green, 3; blue, 3 }  ,draw opacity=1 ][fill={rgb, 255:red, 252; green, 3; blue, 3 }  ,fill opacity=1 ][line width=0.75]  (422.48,86.73) .. controls (422.48,85.03) and (421.1,83.66) .. (419.4,83.66) .. controls (417.7,83.66) and (416.32,85.03) .. (416.32,86.73) .. controls (416.32,88.43) and (417.7,89.81) .. (419.4,89.81) .. controls (421.1,89.81) and (422.48,88.43) .. (422.48,86.73) -- cycle ;
\draw   (408.65,86.73) .. controls (408.65,80.8) and (413.46,75.98) .. (419.4,75.98) .. controls (425.34,75.98) and (430.15,80.8) .. (430.15,86.73) .. controls (430.15,92.67) and (425.34,97.48) .. (419.4,97.48) .. controls (413.46,97.48) and (408.65,92.67) .. (408.65,86.73) -- cycle ;
\draw [color=blue  ,draw opacity=1 ]   (419.4,97.48) -- (416.5,102.25) ;

\draw (291.5,175.75) node [anchor=north west][inner sep=0.75pt]  [font=\footnotesize] [align=left] {(i)};
\draw (322,63) node [anchor=north west][inner sep=0.75pt]  [font=\scriptsize] [align=left] {$\displaystyle 1$};
\draw (296.5,121.5) node [anchor=north west][inner sep=0.75pt]  [font=\scriptsize] [align=left] {$\displaystyle 2$};
\draw (312.67,152.33) node [anchor=north west][inner sep=0.75pt]  [font=\scriptsize] [align=left] {$\displaystyle 3$};
\draw (429,63) node [anchor=north west][inner sep=0.75pt]  [font=\scriptsize] [align=left] {$\displaystyle 1$};
\draw (403.5,121.5) node [anchor=north west][inner sep=0.75pt]  [font=\scriptsize] [align=left] {$\displaystyle 2$};
\draw (429,63) node [anchor=north west][inner sep=0.75pt]  [font=\scriptsize] [align=left] {$\displaystyle 1$};
\draw (403.5,121.5) node [anchor=north west][inner sep=0.75pt]  [font=\scriptsize] [align=left] {$\displaystyle 2$};
\draw (424.03,98.25) node [anchor=north west][inner sep=0.75pt]  [font=\scriptsize] [align=left] {$\displaystyle 3$};
\draw (391.5,176) node [anchor=north west][inner sep=0.75pt]  [font=\footnotesize] [align=left] {(ii)};

\end{tikzpicture}
    \caption{Configurations with both leaf and ladybug but does not give an index 3 resolution configuration.}
    \label{fig:non_ladybug_leaf_config}
\end{figure}
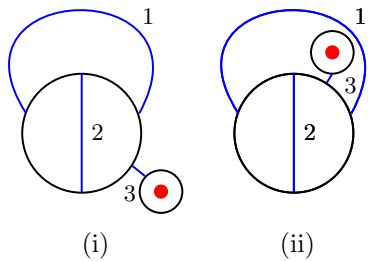
\end{center}    
\end{remark}

\begin{corollary}
Suppose $D$ is a co-leaf. Then $\partial_{exp}\calM(D,x,y)$ is a disjoint union of $6$-cycles.
\end{corollary}

\begin{proof}
The graphs, $\partial_{exp}\calM(D,x,y)$ and $\partial_{exp}\calM_*(D^*,y^*,x^*)$ have the same number of vertices and same parity of number of cycles. So, when $D$ has no ladybug configurations $\partial_{exp}\calM(D,x,y)$ is a $6$-cycle and when $D$ has ladybug configuration $\partial_{exp}\calM(D,x,y)$ is a disjoint union of two $6$-cycles.
\end{proof}

\begin{center}
\begin{figure}[htp]
    \centering
    \input{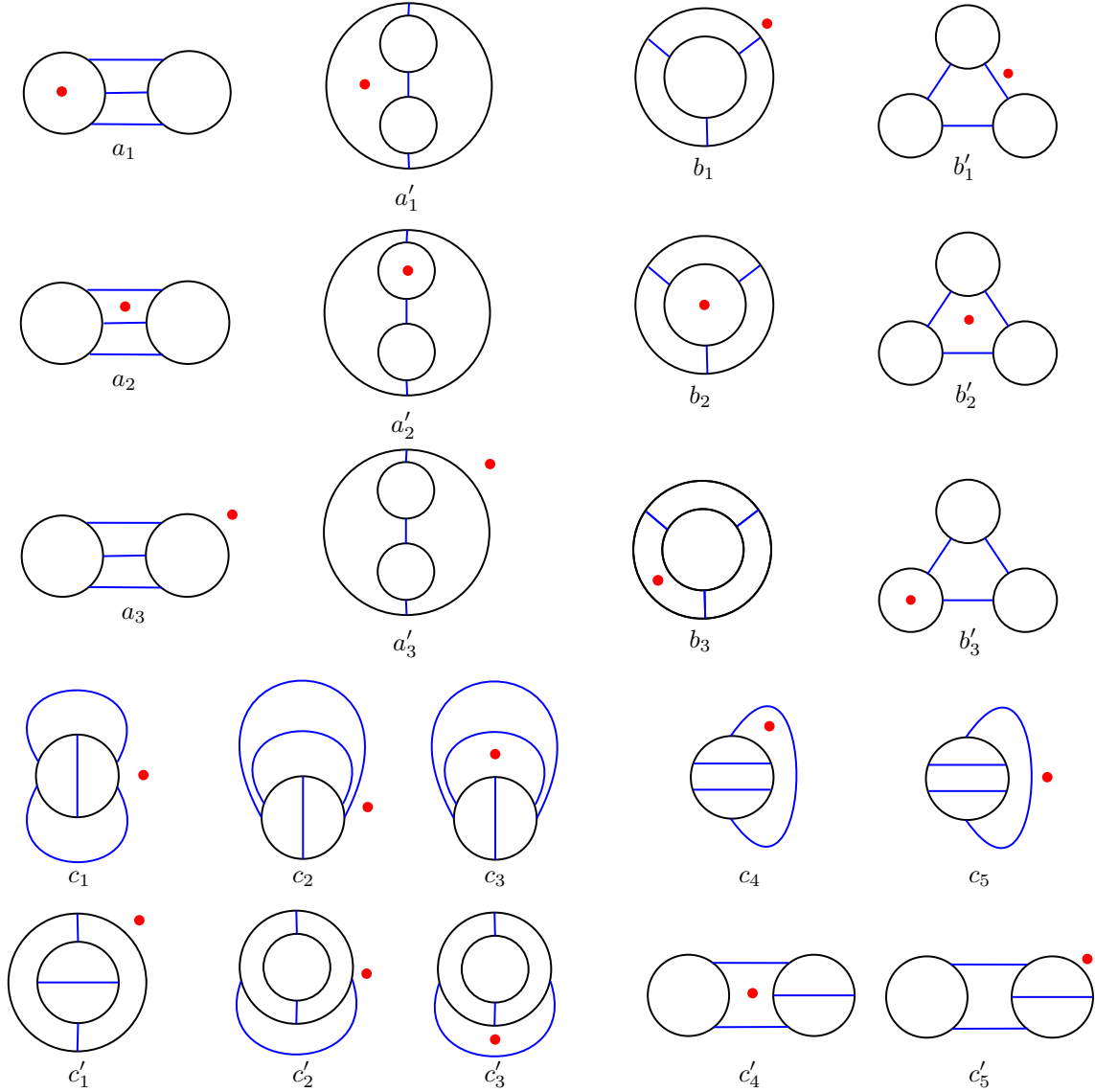}
    \caption{up to isotopy in $A$ and reordering of arcs, these are the only basic index $3$ resolution configurations without leaves or co-leaves. The configurations  $a'_i$, $b'_i$ and $c'_i$ are dual to the configurations  $a_i$, $b_i$, and $c_i$ respectively.} 
    \label{fig:configurations}
\end{figure}   
\end{center}

\begin{lemma}
    Up to isotopy in $A$ and reordering of arcs, the $22$  basic resolution configurations of Figure \ref{fig:configurations} are the only ones with no leaves or co-leaves.
\end{lemma}

\begin{lemma}
    If $D$ is any of the basic resolution configurations depicted in Figure \ref{fig:configurations}, then $\partial_{exp}\calM(D,x,y)$ is a disjoint union of $6$-cycles.
\end{lemma}
\begin{proof}
For $D$ a configuration of the type $a_1$, $a_2$, $a_3$, $b_1$, $b_2$ and $b_3$, in any maximal chain 
\begin{align*}
    c = [(D,y)\prec' (Z_1,z_1)\prec' (Z_2,z_2)\prec' (s(D),x)],
\end{align*}
for a given choice of $D$, $x$ and $y$ the labeled resolution configurations $(Z_1,z_1)$ and $(Z_2,z_2)$ in between obtained after surgery on $(D,y)$ are uniquely determined. This is true irrespective of where the axis is, so $\partial_{exp}\calM(D,x,y)$ is a $6$-cycle.

When $D$ is the configuration of the type $c_1$, $c_2$ and $c_5$, it can be noticed that there is a ladybug configuration of type (i) within these configurations and in any maximal chain 
\begin{align*}
    c = [(D,y)\prec' (Z_1,z_1)\prec' (Z_2,z_2)\prec' (s(D),x)],
\end{align*}
grading reasons forces $y=w_{+}$ and $x=w_{-}w_{-}$. The axis does not interact with the configurations, so $z_i$'s are some combination of $w_{+}$ and $w_{-}$. Therefore the same proof for \cite[Lemma 5.17]{LipSar2014} works, by replacing the labels $x_{+}$ and $x_{-}$ with $w_{+}$ and $w_{-}$ respectively. Hence, $\partial_{exp}\calM(D,x,y)$ is a disjoint union of two $6$-cycles.  

When $D$ is the configuration of the type $c_3$ and $c_4$, it can be noticed that there  is a ladybug configuration of type (ii) within these configurations, so we will use ladybug matching described in the previous section. Let us first label the arcs $R_1$, $R_2$, and $R_3$ in these configurations, see Figure \ref{fig:arc_labels}.

\vspace{-1cm}
\begin{center}
    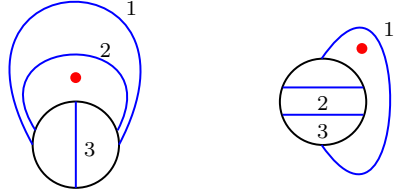
\begin{figure}[htp]
        \centering
        \tikzset{every picture/.style={line width=0.75pt}} 

\begin{tikzpicture}[x=0.75pt,y=0.75pt,yscale=-1,xscale=1]

\draw   (43.55,97.58) .. controls (43.55,85.55) and (53.3,75.79) .. (65.33,75.79) .. controls (77.36,75.79) and (87.11,85.55) .. (87.11,97.58) .. controls (87.11,109.61) and (77.36,119.36) .. (65.33,119.36) .. controls (53.3,119.36) and (43.55,109.61) .. (43.55,97.58) -- cycle ;
\draw [color=blue  ,draw opacity=1 ]   (65.33,75.79) -- (65.33,119.36) ;
\draw [color=blue  ,draw opacity=1 ]   (45,90.32) .. controls (13.05,40.94) and (113.98,37.67) .. (86.39,90.32) ;
\draw [color=blue  ,draw opacity=1 ]   (43.55,97.58) .. controls (-9.46,2.09) and (137.94,0.64) .. (87.11,97.58) ;
\draw  [color={rgb, 255:red, 252; green, 3; blue, 3 }  ,draw opacity=1 ][fill={rgb, 255:red, 252; green, 3; blue, 3 }  ,fill opacity=1 ][line width=0.75]  (67.51,63.72) .. controls (67.51,62.49) and (66.51,61.49) .. (65.27,61.49) .. controls (64.04,61.49) and (63.04,62.49) .. (63.04,63.72) .. controls (63.04,64.95) and (64.04,65.95) .. (65.27,65.95) .. controls (66.51,65.95) and (67.51,64.95) .. (67.51,63.72) -- cycle ;
\draw   (168.24,75.94) .. controls (168.24,63.91) and (177.99,54.16) .. (190.02,54.16) .. controls (202.05,54.16) and (211.8,63.91) .. (211.8,75.94) .. controls (211.8,87.97) and (202.05,97.72) .. (190.02,97.72) .. controls (177.99,97.72) and (168.24,87.97) .. (168.24,75.94) -- cycle ;
\draw [color=blue  ,draw opacity=1 ]   (169.69,68.68) -- (210.35,68.68) ;
\draw [color=blue  ,draw opacity=1 ]   (169.69,82.47) -- (210.35,82.47) ;
\draw [color=blue  ,draw opacity=1 ]   (189.29,54.16) .. controls (235.76,-13.37) and (235.04,163.07) .. (189.29,97.72) ;
\draw  [color={rgb, 255:red, 252; green, 3; blue, 3 }  ,draw opacity=1 ][fill={rgb, 255:red, 252; green, 3; blue, 3 }  ,fill opacity=1 ][line width=0.75]  (211.86,49.02) .. controls (211.86,47.78) and (210.86,46.78) .. (209.62,46.78) .. controls (208.39,46.78) and (207.39,47.78) .. (207.39,49.02) .. controls (207.39,50.25) and (208.39,51.25) .. (209.62,51.25) .. controls (210.86,51.25) and (211.86,50.25) .. (211.86,49.02) -- cycle ;

\draw (89,23.5) node [anchor=north west][inner sep=0.75pt]  [font=\tiny] [align=left] {$\displaystyle 1$};
\draw (219,34.5) node [anchor=north west][inner sep=0.75pt]  [font=\tiny] [align=left] {$\displaystyle 1$};
\draw (76,45) node [anchor=north west][inner sep=0.75pt]  [font=\tiny] [align=left] {2};
\draw (68,95) node [anchor=north west][inner sep=0.75pt]  [font=\tiny] [align=left] {3};
\draw (122,152) node [anchor=north west][inner sep=0.75pt]   [align=left] {};
\draw (185.5,71.5) node [anchor=north west][inner sep=0.75pt]  [font=\tiny] [align=left] {2};
\draw (185.5,86) node [anchor=north west][inner sep=0.75pt]  [font=\tiny] [align=left] {3};

\end{tikzpicture}
        \vspace{-1.2 cm}
        \caption{Configurations $c_3$ and $c_4$ with arcs labelled $R_1$, $R_2$ and $R_3$.}
        \label{fig:arc_labels}
    \end{figure}
\end{center}

\vspace{-0.8cm}
For the configuration $c_3$, the cube of resolution is shown below in Figure \ref{fig:b3_config}. 
\begin{figure}[htp]
    \centering
    \input{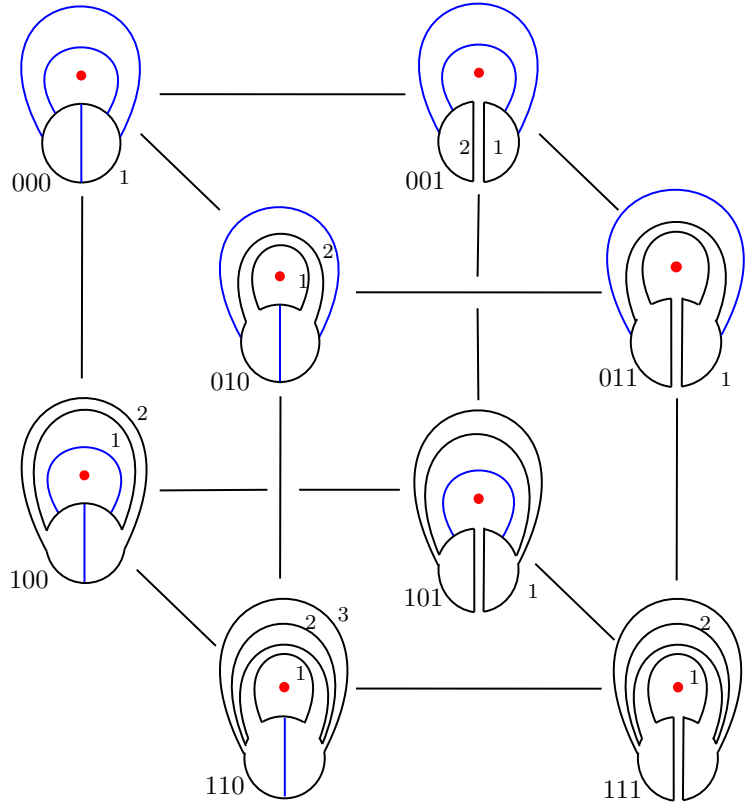}
    \caption{The cube of resolutions for configuration $c_3$.}
    \label{fig:b3_config}
\end{figure}

The vertices in $\partial_{exp}\calM(D,x,y)$ are:
\begin{align*}
    v_1 &= [(D,w_{+})\prec' (s_{\{R_1\}}(D),v_{+}v_{-}) \prec'(s_{\{R_1,R_2\}}(D),v_{+}w_{-}v_{-}) \prec' (s(D),w_{-}w_{-})]\\
    v_2 &= [(D,w_{+})\prec' (s_{\{R_1\}}(D),v_{+}v_{-}) \prec'(s_{\{R_1,R_3\}}(D),w_{-}) \prec' (s(D),w_{-}w_{-})]\\
    v_3 &= [(D,w_{+})\prec' (s_{\{R_1\}}(D),v_{-}v_{+}) \prec'(s_{\{R_1,R_2\}}(D),v_{-}w_{-}v_{-}) \prec' (s(D),w_{-}w_{-})]\\
    v_4 &= [(D,w_{+})\prec' (s_{\{R_1\}}(D),v_{-}v_{+}) \prec'(s_{\{R_1,R_3\}}(D),w_{-}) \prec' (s(D),w_{-}w_{-})]\\
    v_5 &= [(D,w_{+})\prec' (s_{\{R_2\}}(D),v_{+}v_{-}) \prec'(s_{\{R_1,R_2\}}(D),v_{+}w_{-}v_{-}) \prec' (s(D),w_{-}w_{-})]\\
    v_6 &= [(D,w_{+})\prec' (s_{\{R_2\}}(D),v_{+}v_{-}) \prec'(s_{\{R_2,R_3\}}(D),w_{-}) \prec' (s(D),w_{-}w_{-})]\\
    v_7 &= [(D,w_{+})\prec' (s_{\{R_2\}}(D),v_{-}v_{+}) \prec'(s_{\{R_1,R_2\}}(D),v_{-}w_{-}v_{-}) \prec' (s(D),w_{-}w_{-})]\\
    v_8 &= [(D,w_{+})\prec' (s_{\{R_2\}}(D),v_{-}v_{+}) \prec'(s_{\{R_2,R_3\}}(D),w_{-}) \prec' (s(D),w_{-}w_{-})]\\
    v_9 &= [(D,w_{+})\prec' (s_{\{R_3\}}(D),w_{+}w_{-}) \prec'(s_{\{R_1,R_3\}}(D),w_{-}) \prec' (s(D),w_{-}w_{-})]\\
    v_{10} &= [(D,w_{+})\prec' (s_{\{R_3\}}(D),w_{+}w_{-}) \prec'(s_{\{R_2,R_3\}}(D),w_{-}) \prec' (s(D),w_{-}w_{-})]\\
    v_{11} &= [(D,w_{+})\prec' (s_{\{R_3\}}(D),w_{-}w_{+}) \prec'(s_{\{R_1,R_3\}}(D),w_{-}) \prec' (s(D),w_{-}w_{-})]\\
    v_{12} &= [(D,w_{+})\prec' (s_{\{R_3\}}(D),w_{-}w_{+}) \prec'(s_{\{R_2,R_3\}}(D),w_{-}) \prec' (s(D),w_{-}w_{-})]
\end{align*}

Lets denote the faces of the cube by a $4$-tuple of vertices $i\in \{0,1\}^3$ and the $n$-th circle in the $i$-th vertex by $Z_{n,i}$. The faces $(000, 010, 011, 001)$ and $(000, 100, 101, 001)$ come from ladybug configurations. The ladybug matchings are given by,

\begin{center}
\begin{tabular}{c | c}
    Faces & Matchings \\
    \hline
     & \\
    $\multirow{2}{*}{(000, 010, 011, 001)}$ & $Z_{1,010} \longleftrightarrow Z_{1,001}$,\\ & $Z_{2,010} \longleftrightarrow Z_{2,001}$ \\
     & \\
    \hline
     & \\
    $\multirow{2}{*}{(000, 100, 101, 001)}$ & $Z_{1,100} \longleftrightarrow Z_{1,001}$,\\
     & $Z_{2,100} \longleftrightarrow Z_{2,001}$\\
     & 
\end{tabular}
\end{center}

The following edges are not part of the ladybug configurations, so they exist independently of the ladybug matching,

\begin{center}
\begin{tabular}{c c c c}
    $v_1-v_2$ & $v_3-v_4$ & $v_5-v_6$ & $v_7-v_8$\\
    $v_9-v_{10}$ & $v_{11}-v_{12}$ & $v_1-v_5$ & $v_3-v_7$ 
\end{tabular}   
\end{center}

The ladybug matchings gives the $4$-other edges,

\begin{center}
\begin{tabular}{c c c c}
    $v_2-v_9$ & $v_4-v_{11}$ & $v_6-v_{10}$ & $v_8-v_{12}$
\end{tabular}   
\end{center}

So, $\partial_{exp}\calM(D,x,y)$ is a disjoint union of two $6$-cycles and the components of  are,

\begin{center}
\begin{tabular}{c c c}
    $v_1-v_2-v_9-v_{10}-v_6-v_5-v_1$ & \text{and} & $v_3-v_4-v_{11}-v_{12}-v_8-v_7-v_3.$
\end{tabular}   
\end{center}

For the configuration $c_4$, the cube of resolution is shown below in Figure \ref{fig:b4_config}. 

\begin{figure}[htp]
    \centering
    \input{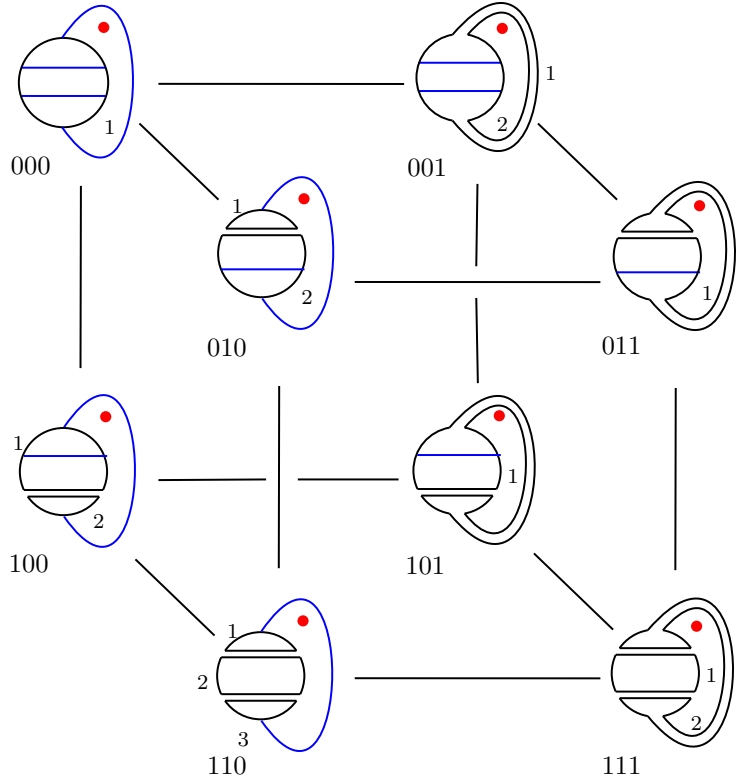}
    \vspace{-1.5cm}
    \caption{The cube of resolutions for configuration $c_4$.}
    \label{fig:b4_config}
\end{figure}

The vertices in $\partial_{exp}\calM(D,x,y)$ are:
\begin{align*}
    v_1 &= [(D,w_{+})\prec' (s_{\{R_1\}}(D),w_{+}w_{-}) \prec'(s_{\{R_1,R_2\}}(D),w_{+}w_{-}w_{-}) \prec' (s(D),w_{-}w_{-})]\\
    v_2 &= [(D,w_{+})\prec' (s_{\{R_1\}}(D),w_{+}w_{-}) \prec'(s_{\{R_1,R_3\}}(D),w_{-}) \prec' (s(D),w_{-}w_{-})]\\
    v_3 &= [(D,w_{+})\prec' (s_{\{R_1\}}(D),w_{-}w_{+}) \prec'(s_{\{R_1,R_2\}}(D),w_{-}w_{-}w_{+}) \prec' (s(D),w_{-}w_{-})]\\
    v_4 &= [(D,w_{+})\prec' (s_{\{R_1\}}(D),w_{-}w_{+}) \prec'(s_{\{R_1,R_3\}}(D),w_{-}) \prec' (s(D),w_{-}w_{-})]\\
    v_5 &= [(D,w_{+})\prec' (s_{\{R_2\}}(D),w_{+}w_{-}) \prec'(s_{\{R_1,R_2\}}(D),w_{+}w_{-}w_{-}) \prec' (s(D),w_{-}w_{-})]\\
    v_6 &= [(D,w_{+})\prec' (s_{\{R_2\}}(D),w_{+}w_{-}) \prec'(s_{\{R_2,R_3\}}(D),w_{-}) \prec' (s(D),w_{-}w_{-})]\\
    v_7 &= [(D,w_{+})\prec' (s_{\{R_2\}}(D),w_{-}w_{+}) \prec'(s_{\{R_1,R_2\}}(D),w_{-}w_{-}w_{+}) \prec' (s(D),w_{-}w_{-})]\\
    v_8 &= [(D,w_{+})\prec' (s_{\{R_2\}}(D),w_{-}w_{+}) \prec'(s_{\{R_2,R_3\}}(D),w_{-}) \prec' (s(D),w_{-}w_{-})]\\
    v_9 &= [(D,w_{+})\prec' (s_{\{R_3\}}(D),v_{+}v_{-}) \prec'(s_{\{R_1,R_3\}}(D),w_{-}) \prec' (s(D),w_{-}w_{-})]\\
    v_{10} &= [(D,w_{+})\prec' (s_{\{R_3\}}(D),v_{+}v_{-}) \prec'(s_{\{R_2,R_3\}}(D),w_{-}) \prec' (s(D),w_{-}w_{-})]\\
    v_{11} &= [(D,w_{+})\prec' (s_{\{R_3\}}(D),v_{-}v_{+}) \prec'(s_{\{R_1,R_3\}}(D),w_{-}) \prec' (s(D),w_{-}w_{-})]\\
    v_{12} &= [(D,w_{+})\prec' (s_{\{R_3\}}(D),v_{-}v_{+}) \prec'(s_{\{R_2,R_3\}}(D),w_{-}) \prec' (s(D),w_{-}w_{-})]
\end{align*}

Here, the faces $(000, 010, 011, 001)$ and $(000, 100, 101, 001)$ come from ladybug configurations. We use the same ladybug matchings as in the above case and it can be checked that $\partial_{exp}\calM(D,x,y)$ is a disjoint union of two $6$-cycles given below.

\begin{center}
\begin{tabular}{c c c}
    $v_1-v_2-v_9-v_{10}-v_6-v_5-v_1$ & \text{and} & $v_3-v_4-v_{11}-v_{12}-v_8-v_7-v_3$
\end{tabular}   
\end{center}

The configurations  $a'_i$, $b'_i$ and $c'_i$ are dual to the configurations  $a_i$, $b_i$ and $c_i$ respectively. So, for $D$ of the type  $a'_i$, $b'_i$ and $c'_i$, the graph $\partial_{exp}\calM(D,x,y)$ is a disjoint union of $6$-cycles.
\end{proof}

\subsection{$n$-dimensional moduli spaces, $n\geq 3$}
\begin{proposition}
There exist spaces $\mathcal{M}(D,x,y)$ and maps $\mathcal{F}$ that satisfy the conditions (RM-1)--(RM-4) for resolution moduli spaces.
\end{proposition}

\subsection{Khovanov skein flow category}
\begin{definition}
The Khovanov skein flow category \( \mathscr{C}_{Sk}(D_L) \) has one object for each labeled resolution configuration of the form \( \mathbf{x} = (D_L(u), x) \), where \( u \in \{0,1\}^{n(D_L)} \). For objects \( \mathbf{x} = (D_L(u), x) \) and \( \mathbf{y} = (D_L(v), y) \) such that \( \mathbf{y} \prec' \mathbf{x} \), the moduli space  
\[
\mathcal{M}_{\mathscr{C}_{Sk}(D_L)}(\mathbf{x}, \mathbf{y}) = \mathcal{M}(D_L(v) \setminus D_L(u), x|, y|)
\]
is a smooth manifold with corners. For all other pairs of objects \( \mathbf{x} \), \( \mathbf{y} \) not related by \( \prec' \), the space \( \mathcal{M}(\mathbf{x}, \mathbf{y}) \) is defined to be empty. The composition maps in \( \mathscr{C}_{Sk}(D_L) \) are induced from the composition maps in the resolution moduli spaces in the annulus.

The flow category grading on the objects of \( \mathscr{C}_{Sk}(D_L) \) is given by the homological grading \( \hgr \). In addition, the objects are equipped with two auxiliary gradings: the quantum grading \( \qgr \) and the homotopical grading \( \fgr \).

There is a cover functor
\[
\mathscr{F} : \mathscr{C}_{Sk}(D_L) \to \mathscr{C}_C(n(D_L))[-n_{-}(D_L)]
\]
defined as follows. On the level of objects, \( \mathscr{F} \) is the forgetful functor:
\[
\mathscr{F}((D_L(v), y)) = v.
\]
On the level of morphisms, \( \mathscr{F} \) is a covering map
\[
\mathscr{F} : \mathcal{M}_{\mathscr{C}_{Sk}(D_L)}((D_L(u), x), (D_L(v), y)) \to \mathcal{M}_{\mathscr{C}_C(n(D_L))}(u, v),
\]
defined as the composition
\[
\mathcal{M}(D_L(v) \setminus D_L(u), x|, y|) 
\xrightarrow{\mathcal{F}} \mathcal{M}_{\mathscr{C}_C(|u|-|v|)}(\overline{1}, \overline{0})
\xhookrightarrow{\mathcal{I}_{u,v}} \mathcal{M}_{\mathscr{C}_C(n(D_L))}(u, v).
\]
The cube flow category \( \mathscr{C}_C(n(D_L)) \) can be regarded as a framed flow category, where the framing is determined by a sign assignment \( s \) on the cube \( C(n(D_L)) \); see \cite[Proposition 4.12]{LipSar2014}. Furthermore, since \( \mathscr{C}_{Sk}(D_L) \) is a cover of the framed flow category \( \mathscr{C}_C(n(D_L)) \), we may also view \( \mathscr{C}_{Sk}(D_L) \) as a framed flow category, where the coherent framing and neat embedding are induced from those on \( \mathscr{C}_C(n(D_L)) \).
\end{definition}

Suppose the (flow category) gradings of all objects in \( \mathscr{C}_{Sk}(D_L) \) lie within the interval \([B, A]\). Let  
\[
N = d_{B} + d_{B+1} + \cdots + d_{A-1} - B,
\]  
where \( \mathbf{d} \) is the sequence specifying the dimensions used in the neat embedding of the flow category \( \mathscr{C}_{Sk}(D_L) \); see \cite[Definition 3.15]{LipSar2014}.

\begin{definition}
The \textit{Khovanov skein space} is the Cohen–Jones–Segal realization \( |\mathscr{C}_{Sk}(D_L)| \) of the framed flow category \( \mathscr{C}_{Sk}(D_L) \); see \cite[Definition 3.23]{LipSar2014}. The \textit{Khovanov skein spectrum} \( \skspec(L) \) is defined as the suspension spectrum of the Khovanov skein space, desuspended by an integer \( N \). By \cite[Lemma 3.31]{LipSar2014}, this spectrum decomposes as a wedge sum over quantum gradings \( q \) and homotopical gradings \( f \):
\[
\skspec(L) = \bigvee_{q,f} \skspecgrad{q}{f}(L).
\]
\end{definition}

\subsection{Invariance of Khovanov skein spectrum}

\begin{theorem}
Let \( L \) be a link diagram, and let \( L' \) be obtained from \( L \) by a sequence of Reidemeister moves \( \Rom{1} \), \( \Rom{2} \), and \( \Rom{3} \) performed in the annulus. Then, for all quantum gradings \( q \) and homotopical gradings \( f \), the skein spectrum \( \skspecgrad{q}{f}(L) \) is stably homotopy equivalent to \( \skspecgrad{q}{f}(L') \).
\end{theorem}

\begin{proof}
    The proof is analogous to that of the invariance of the Khovanov spectrum under Reidemeister moves; see \cite[Proposition 6.1]{LipSar2014}. Suppose \( L' \) is obtained from \( L \) by an allowable Reidemeister \( \Rom{1} \) move in the annulus.
\begin{figure}[htp]
    \centering
    \input{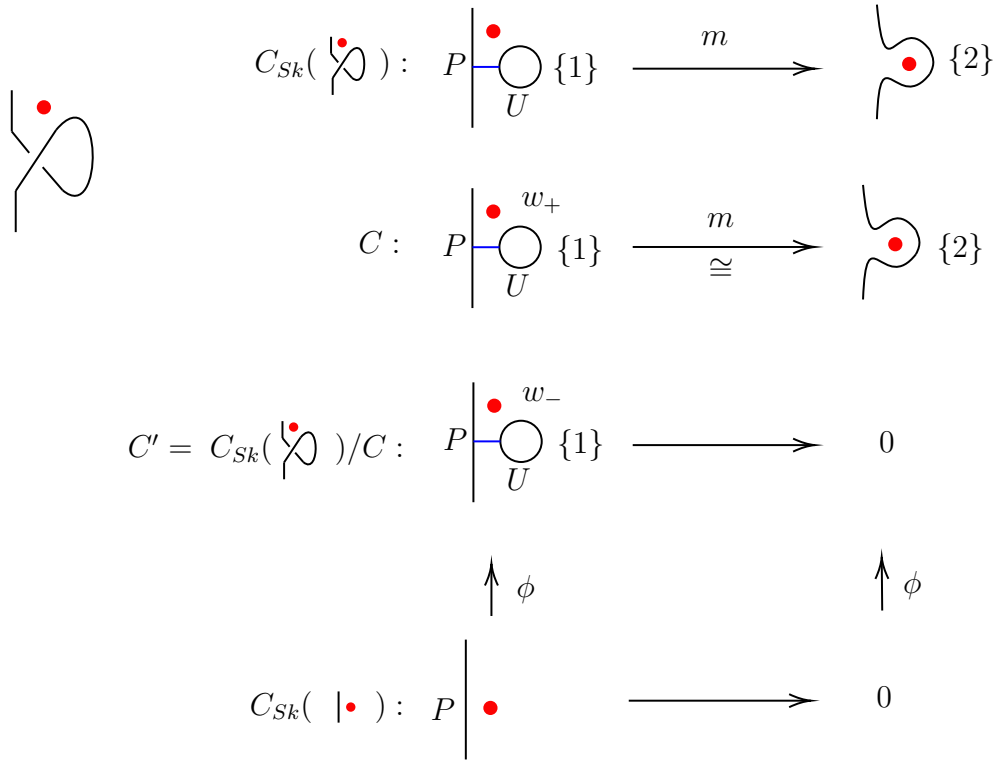}
    \caption{Invariance under Reidemeister moves \Rom{1}.}
    \label{fig:Invariance_under_R1}    
\end{figure}
Note that there exists an acyclic subcomplex \( C \) of \( C_{Sk}(L') \), where the special circle \( U \) is trivial and labeled with \( w_{+} \); see Figure~\ref{fig:Invariance_under_R1}. This subcomplex \( C \) corresponds to an upward closed subcategory \( \mathscr{C}\) of \( \mathscr{C}_{Sk}(L') \), and the quotient complex \( C' = C_{Sk}(L') / C \) corresponds to the complementary downward closed subcategory \( \mathscr{C}' \). There is a cochain isomorphism
\[
\phi: C_{Sk}(L) \to C'
\]
defined by \( \phi(x) = x \otimes w_{-} \), where \( x \) is a labeling on \( P \). The map \( \phi \) preserves both the quantum and homotopical gradings. Moreover, the downward closed subcategory \( \mathscr{C}' \) is isomorphic to \( \mathscr{C}_{Sk}(L) \). Finally, since $C$ is acyclic, hence the inclusion map $|\scrC_{Sk}(L)|\cong |\scrC'|\xhookrightarrow{} |\scrC_{Sk}(L')|$ is a homotopy equivalence for some coherent choices of framing by \cite[Lemma 3.32]{LipSar2014}. The key idea in this proof is that, since we consider only allowable Reidemeister moves within the annulus, the special circle \( U \) is always trivial. For a similar reason, the invariance under Reidemeister moves \( \Rom{2} \) and \( \Rom{3} \) in the annulus can also be established.
\end{proof}

For a framed flow category \( \mathscr{C} \), the associated cochain complex \( (C^{*}(\mathscr{C}), \partial) \) is defined as follows; see \cite[Definition 3.20]{LipSar2014}. The cochain groups \( C^{n}(\mathscr{C}) \) are free \(\mathbb{Z}\)-modules generated by objects in \( \Ob(\mathscr{C}) \) with flow category grading \( n \). For \( x,y \in \Ob(\mathscr{C}) \) satisfying \(\mathrm{gr}(x) = \mathrm{gr}(y) + 1 \), the coefficient \( \langle \partial y, x \rangle \) is given by the signed count of the moduli space \( \mathcal{M}_{\mathscr{C}}(x,y) \).

Similarly, working with \(\mathbb{Z}_2\)-coefficients, we define the cochain complex \( (\overline{C^*(\mathscr{C})}, \overline{\partial}; \mathbb{Z}_2) \) where the groups \( \overline{C^n(\mathscr{C})} \) are free \(\mathbb{Z}_2\)-modules generated by objects of grading \( n \). For \( x,y \in \Ob(\mathscr{C}) \) with \(\mathrm{gr}(x) = \mathrm{gr}(y) + 1 \), the coefficient \( \langle \overline{\partial} y, x \rangle \) is the mod 2 count of \( \mathcal{M}_{\mathscr{C}}(x,y) \).

\begin{lemma}\label{lemma 3.4}
    For a framed flow category $\mathscr{C}$, if each nonempty $0$-dimensional moduli space $\mathcal{M}_{\mathscr{C}}(x,y)$ consists of one point, then the reduced cellular cochain complex $\widetilde{C}^{*}(|\mathscr{C}|;\mathbb{Z}_{2})[-N]$ is isomorphic to the associated cochain complex $(\overline{C^{*}(\mathscr{C})},\overline{\partial};\mathbb{Z}_{2})$ where $[\cdot]$ denotes the degree shift operator.
\end{lemma}
\begin{proof}
   For $x,y\in \text{Ob}(\mathscr{C})$, such that $\text{gr}(x)=\text{gr}(y)+1$, the cellular boundary map $\pi: \partial C(x)\to C(y)/ \partial C(y)$ has degree $\#  \pi^{-1}(p)$ for any $p\in \text{interior}\big(C(y)\big)$. But since we have a homeomorphism $C_{y}(x) \cong C(y) \times \mathcal{M}_{\mathscr{C}}(x,y) $, the signed count of $\#  \pi^{-1}(p)$ is equal to $(-1)^{K}\times \#\mathcal{M}_{\mathscr{C}}(x,y)$ for some $K\in \mathbb{Z}$; see the proof of \cite[Lemma 3.24]{LipSar2014}. If we assume that $\mathcal{M}_{\mathscr{C}}(x,y)$ consists of a single point (if nonempty), then the coefficient of $x$ in $\overline{\partial} y$ is $<\overline{\partial}y,x> = \overline{1}\in \mathbb{Z}_{2}$. On the other hand, the degree of the map $\pi$ is also $\overline{(-1)^k} =\overline{1}\in \mathbb{Z}_{2}$. Now this lemma follows because $\Hom(\mathbb{Z}, \mathbb{Z}_{2})\cong \mathbb{Z}_{2}$.
\end{proof}

For two objects \(\mathbf{x}, \mathbf{y}\) in \(\mathscr{C}_{Sk}(D_L)\) satisfying \(\mathbf{y} \prec' \mathbf{x}\), the homological grading satisfies \(\hgr(\mathbf{y}) < \hgr(\mathbf{x})\), while the quantum and homotopical gradings are preserved: \(\qgr(\mathbf{y}) = \qgr(\mathbf{x})\) and \(\fgr(\mathbf{y}) = \fgr(\mathbf{x})\). Consequently, for fixed gradings \(j\) and \(k\), the full subcategory \(\mathscr{C}_{Sk}^{j,k}(D_L)\) of \(\mathscr{C}_{Sk}(D_L)\) is defined by objects \(\mathbf{x}\) satisfying \(\qgr(\mathbf{x}) = j\) and \(\fgr(\mathbf{x}) = k\). This subcategory inherits a framed flow category structure with neat embeddings and framings induced from the framed flow category \(\mathscr{C}_{Sk}(D_L)\). Thus, we have the following corollary.
\begin{corollary}\label{corollary 3.11}
    The reduced cellular cochain complex $\widetilde{C}^{*}(|\mathscr{C}_{Sk}^{j,k}(D_{L})|;G)[-N]$ is isomorphic to $ (C_{Sk}^{*,j,k}(D_{L}), \partial_{Sk}^{*,j,k} )$ when $G=\mathbb{Z}$ and isomorphic to $(\overline{C_{Sk}^{*,j,k}(D_{L})},\overline{\partial_{Sk}^{*,j,k}};\mathbb{Z}_{2})$ when $G=\mathbb{Z}_{2}$.
\end{corollary}
\begin{proof}
     When $G=\mathbb{Z}$, the associated cochain complex $(C^{*}(\mathscr{C}_{Sk}^{j,k}(D_{L})),\partial)$ is the same cochain complex as $(C_{Sk}^{*,j,k}(D_{L}), \partial_{Sk}^{*,j,k} )$ since all nonempty zero-dimensional moduli spaces in Khovanov skein flow category consist of single points, and now we use \cite[Lemma 3.24]{LipSar2014}. The $G=\mathbb{Z}_{2}$ case follows from applying Lemma \ref{lemma 3.4}.
\end{proof}

\section{Transverse invariant as extreme Khovanov skein  Spectrum} \label{section_4}
Given an oriented link diagram \( D_L \) of an oriented link \( L \subset A \times I \), there is a natural embedding $L \hookrightarrow A \times I \hookrightarrow S^3$. A labeled resolution configuration \( (D_L(\alpha), y) \in C_{Kh}(D_L) \), originally defined in the sphere \( S^2 \), can be reinterpreted in the annular setting by distinguishing between trivial and nontrivial circles. Specifically, trivial circles—those bounding disks in the annulus—are labeled with \( w_\pm \) in place of \( x_\pm \), while nontrivial circles—those representing nontrivial elements in \( H_1(A; \mathbb{Z}) \)—are labeled with \( v_\pm \). This refinement allows \( (D_L(\alpha), y) \) to be understood as an element of the annular Khovanov complex.
 \par
Moreover, the Khovanov differential \( \partial_{Kh} \) increases the homological grading by 1, preserves the quantum grading, and decomposes as $\partial_{Kh} = \partial_{Kh,0} + \partial_{Kh,-2}$, where \( \partial_{Kh,0} \) preserves the homotopical grading and \( \partial_{Kh,-2} \) lowers the homotopical grading by 2; see the proof of \cite[Lemma 2.3]{Rob2013}.
\par
The homotopical grading induces a filtration \( \mathcal{F}^j \) on the complex  
$C_{Kh}^{*,j}(D_L) = \bigoplus_i C_{Kh}^{i,j}(D_L)$ for each fixed quantum grading \( j \), given by
\[
0 \subset \mathcal{F}^j_{f_{\min}(D_L, j)}(D_L) \subset \mathcal{F}^j_{f_{\min}(D_L, j) + 2}(D_L) \subset \cdots \subset C_{Kh}^{*,j}(D_L),
\]
where \( \mathcal{F}^j_f(D_L) \) denotes the subcomplex generated by all labeled resolution configurations \( (D_L(\alpha), y) \in C_{Kh}^{*,j}(D_L) \) whose homotopical grading is less than or equal to \( f \). In \cite{LipSar2014}, Lipshitz and Sarkar constructed a spectrum \( \khspecgrad{j}(L) \), referred to as the \emph{Khovanov spectrum}, associated to an oriented link \( L \subset S^3 \), for each quantum grading \( j \).

\begin{theorem} \label{map-between-spectra}
    For an oriented closed braid diagram $B_{L}\subset A$ with $b(B_{L})$ number of strands representing $L\subset A\times I$, and for each quantum grading $j$, there is a map 
    $$
    \Psi^{j}(B_{L}):\khspecgrad{j}(B_{L}) \to \skspecgrad{j}{f_{\text{min}}(B_{L},j)}(B_{L}),
    $$ 
    such that the induced map on the reduced cohomology 
    $$
    \Psi^{j}(B_{L})^{*}: \widetilde{H}^{i}\big(\skspecgrad{j}{f_{\text{min}} (B_{L},j)}(B_{L}); G\big) \to \widetilde{H}^{i}\big(\khspecgrad{j}(B_{L});G\big)
    $$
    is the same map as $\jmath^{*}: H_{Sk}^{i,j,f_{\text{min}} (B_{L},j)}(L;G) \to H_{Kh}^{i,j}(L;G)$ induced from the embedding $\jmath:  A \times I \hookrightarrow S^{3}$, for $G= \mathbb{Z}_{2}$ or $\mathbb{Z}$.\par
    In particular, when $j=sl(L)$, we get $\Psi^{sl(L)}(B_{L}): \khspecgrad{sl(L)}(B_{L}) \to \skspecgrad{sl(L)}{-b(B_{L})}(B_{L}) =\mathbb{S}$ such that the induced map $\Psi^{sl(L)}(B_{L})^{*}: G\cong H_{Sk}^{0,sl(L), -b(B_{L})}(L;G) \to H_{Kh}^{0,sl(L)}(L;G)$ satisfies $\Psi^{sl(L)}(B_{L})^{*}(\psi_{Sk}(B_{L})) = \psi_{Kh}(B_{L})$, for $G= \mathbb{Z}_{2}$ or $\mathbb{Z}$.
\end{theorem}
\begin{proof}
    Recall, that for each quantum grading $j$, the Khovanov spectrum $\khspecgrad{j}(L)$ is the $N$ times formal desuspension of a suspension spectrum of a CW complex $|\mathscr{C}_{Kh}^{j}(B_{L})|$ which is the Cohen-Jones-Segal realization of the framed flow category $\mathscr{C}_{Kh}^{j}(B_{L})$; see \cite{LipSar2014}. Let $\mathscr{C}^{j}_{u}(B_{L})$ denote the subcategory of the framed flow category $\mathscr{C}_{Kh}^{j}(B_{L})$ containing objects from the collection $\big\{(B_{L}(\alpha),y): (B_{L}(\alpha),y)\in \mathcal{F}_{f_{\text{min}} (B_{L},j)}^{j}(B_{L})\big\}$, and all morphisms between them. Then  $\mathscr{C}^{j}_{u}(B_{L})$ is a full, upward closed subcategory, meaning for any two objects $\textbf{y}=(B_{L}(\alpha),y) $, and $ \textbf{x}=(B_{L}(\beta),x)$ in $\mathscr{C}_{Kh}^{j}(B_{L})$, satisfying $\Hom_{\mathscr{C}_{Kh}^{j}(B_{L})}(\textbf{x},\textbf{y})\neq \emptyset$, if $\textbf{y}\in \Ob (\mathscr{C}^{j}_{u}(B_{L}))$ then $\textbf{x} \in \Ob (\mathscr{C}^{j}_{u}(B_{L}))$; see \cite[Definition 3.29]{LipSar2014}.\par
    Recall from \cite[Definition 3.23]{LipSar2014} that the CW complex \( |\mathscr{C}_{Kh}^{j}(B_L)| \) has a cell \( C(\mathbf{x}) \) associated to each object \( \mathbf{x} \) in the framed flow category \( \mathscr{C}_{Kh}^{j}(B_L) \). By collapsing all cells \( C(\mathbf{x}) \) for which \( \mathbf{x} \notin \Ob(\mathscr{C}^{j}_{u}(B_L)) \), we obtain a quotient map  
    \[
    |\mathscr{C}_{Kh}^{j}(B_L)| \to |\mathscr{C}^{j}_{u}(B_L)|.
    \]
    After formally desuspending both sides \( N \)-times, this induces a map of spectra:
    \[
    \Psi^{j}(B_L): \khspecgrad{j}(B_L) \to \skspecgrad{j}{f_{\min}(B_L, j)}(B_L).
    \]
    Now the rest of the statements follows from the Corollary \ref{corollary 3.11}, and Lemma \ref{lemma 2.2}.
\end{proof}
We conclude this section with an observation related to the work of J. González-Meneses, P. M. G. Manchón, and M. Silvero \cite{MenesManchSilv2018}. In their work, they constructed a spectrum associated to the \emph{extreme Khovanov cohomology}. Subsequently, in \cite{MorSilv2018}, Federico Cantero Morán and Marithania Silvero showed that this extreme Khovanov spectrum is stably homotopy equivalent to the Khovanov spectrum of Lipshitz and Sarkar at the extreme quantum grading.

Given a link diagram \( B_L \subset A \subset \mathbb{R}^2 \), we define the \emph{extreme quantum grading} \( j_{\min}(B_L) \) as
\[
j_{\min}(B_L) = \min \left\{ j \mid \qgr(\alpha, x) = j \right\}.
\]
Let \( \overline{0}^{-} \) denote the labeled resolution configuration \( (B_L(\overline{0}), x) \), where each circle is labeled with \( x_{-} \). Let \( C(\alpha) \) denote the number of circles in the resolution configuration associated to \( \alpha \).

We now recall the following proposition; see \cite{MenesManchSilv2018}.

\begin{proposition}[\cite{MenesManchSilv2018}, Proposition 4.1]
 In the above setting, $j_{min} = \qgr(\overline{0}^{-})$  and $\qgr(\alpha, x) =
j_{min}$  if and only if $(\alpha, x) \in S_{min}$, where
\begin{align*}
S_{min} &= \big\{(\alpha, x)\mid |C(\alpha)| = |C(\overline{0})| + |\alpha| \,  \text{and }\, x =  x_{-}^{\otimes |C(\alpha)|}\big\}.   
\end{align*}
In particular, $j_{min} = n_{+} - 2n_{-} - |C(\overline{0})|$. 
\end{proposition}
This leads us to an immediate corollary.
\begin{corollary}\label{transverse-invariant-in-extreme-state}
    For a closed braid diagram \( B_L \) representing a transverse link \( L \), the labeled resolution configuration \( (B_L^o, x^o) = \widehat{\psi}(L) \) lies in \( S_{\min} \). In this case, the extreme quantum grading coincides with the self-linking number, that is, $j_{\min} = sl(L)$.
\end{corollary}
\begin{proof}
    Recall from Section~\ref{section 2} that \( B_L^o = B_L(\alpha^o) \) denotes the oriented resolution of the braid diagram \( B_L \), where \( |\alpha^o| = n_{-}(B_L) \). If \( \alpha_i^o = 1 \) for some \( i \in \{1, \dots, n(B_L)\} \), then the \( i \)-th crossing must be a negative crossing. Replacing the 1-smoothing at the \( i \)-th crossing with a 0-smoothing yields a new resolution configuration \( B_L(\beta) \), where \( \beta_j = \alpha_j^o \) for all \( j \neq i \), and \( \beta_i = 0 \). This change reduces the number of circles by one, i.e., \( C(\beta) = C(\alpha^o) - 1 \).

    If \( \beta_j = 1 \) for some \( j \), we repeat the same procedure, switching the 1-smoothing to a 0-smoothing at the \( j \)-th crossing. Iterating this process eventually leads to the all-zero state \( \overline{0} \). Each step reduces the number of circles by one, and there are exactly \( |\alpha^o| = n_{-}(B_L) \) such steps. Therefore, $|C(\alpha^o)| = |C(\overline{0})| + |\alpha^o|$, as claimed.
\end{proof}
Analogous to the notion of extreme quantum grading in Khovanov homology, we define the pair \((j_{\min}, f_{\min}(B_L))\) as the \textit{extreme grading} in Khovanov skein homology. Here, \( j_{\min} = sl(L) \), by Corollary~\ref{transverse-invariant-in-extreme-state}, represents the extreme quantum grading, and \( f_{\min}(B_L) = -b \), by Lemma~\ref{lemma 2.2}, represents the extreme homotopical grading. Thus, we obtain the following corollary.
\begin{corollary}\label{corollary 4.4}
    For a closed braid diagram $B_{L}$ representing an transverse link $L$, the extreme Khovanov skein spectrum $\skspecgrad{j_{min}}{f_{\text{min}}(B_{L})}(B_{L}) = \mathbb{S}$ is the sphere spectrum.
\end{corollary}
\begin{remark}
    Using corollary \ref{corollary 4.4}, we have $\Psi^{sl(L)}(B_{L})\in \pi_{s}^{0}(\khspecgrad{sl(L)}(L))$ which is the cohomotopy transverse invariant described in \cite{LipSarTransverse}.
\end{remark}

\section{Example}
Let us consider the following closed braid representation of a two-crossing unknot, $B_{L}\subset A$ in the Figure \ref{fig:example-hypercube}.
Observe that the crossing labelled $1$ is a positive crossing and the one labelled $2$ is a negative crossing. 
\begin{figure}[htp]
    \centering
    \input{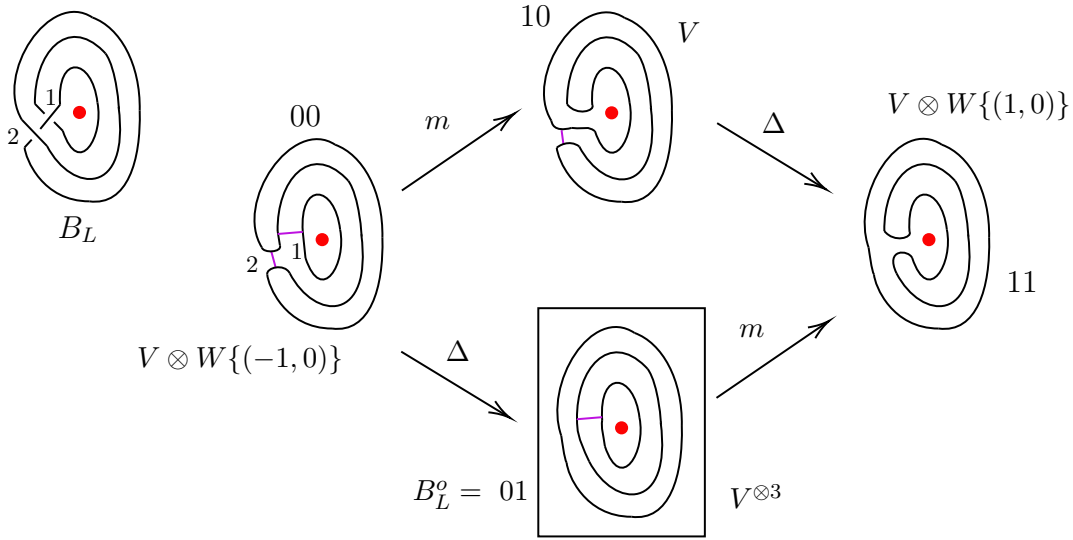}
    \caption{Hypercube of resolutions of a closed braid diagram in annulus.}
    \label{fig:example-hypercube}    
\end{figure}
The oriented resolution $B^{\circ}_{L}$ corresponds to the state $01 \in \{0,1\}^{2}$. Note that $|Z(B^{\circ}_{L})| = 3 = b$, where $b$ is the number strands and the self linking number $sl(B_{L})= -b + n_{+}-n_{-}= -3+1-1=-3$. The labeled resolution configuration $\widehat{\psi}_{Sk}(B_{L}) = (01, v_{-}^{\otimes 3})\in C_{Sk}^{0,-3,-3}(B_{L})$ corresponds to the transverse invariant. 
\begin{figure}[htp]
    \centering
    \input{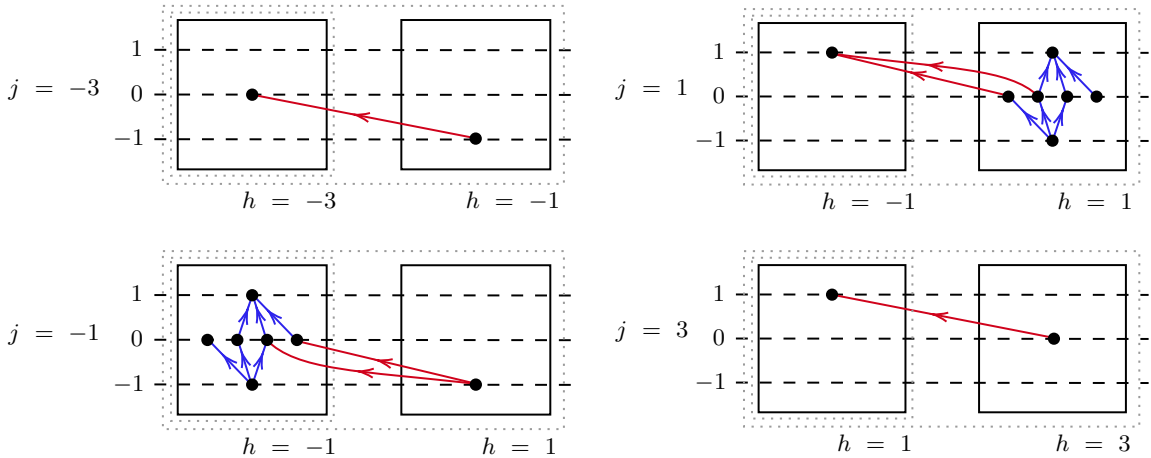}
    \caption{Khovanov chain complex and skein flow category}
    \label{fig:example-skein-flow}    
\end{figure}
Note all generators of $C_{Sk}(B_{L})$ have a homological grading of either $-1$, $0$, or $1$. In the Figure \ref{fig:example-skein-flow}, we gave the Khovanov chain complex of $B_{L}$. As described in Section \ref{section_4}, the Khovanov boundary maps $\partial_{Kh}$ has two components, $\partial_{Kh,0}$ which preserves the homotopical grading is denoted by blue arrows and $\partial_{Kh,-2}$ which decreases the homotopical grading by $2$ are denoted by red arrows in Figure \ref{fig:example-skein-flow}. Note that $\partial_{Sk} = \partial_{Kh,0}$.

For a fixed $j$, the generators of the lowest homotopical grading forms an upward closed full subcategory $\scrC^{j}_{u}(B_{L})$ of $\scrC^{j}_{Kh}(B_{L})$, see Figure \ref{fig:example-maps} for $j=-1$ case. Then we have the quotient map $|\scrC_{Kh}^{j}(B_{L})|\to |\scrC_{u}^{j}(B_{L})|$ by quotienting the cells corresponding to the complementary downward closed subcategory of $\scrC_{u}^{j}(B_{L})$. 
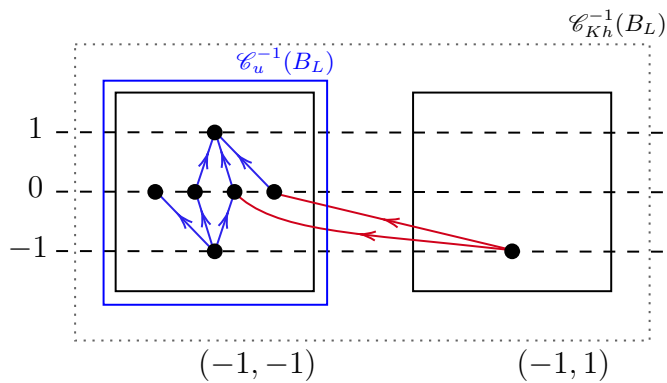
\begin{figure}[htp]
    \centering
    \tikzset{every picture/.style={line width=0.75pt}} 

\begin{tikzpicture}[x=0.75pt,y=0.75pt,yscale=-1,xscale=1]

\draw   (200.5,80) -- (300.43,80) -- (300.43,180.25) -- (200.5,180.25) -- cycle ;
\draw   (350.5,80) -- (450.43,80) -- (450.43,180.25) -- (350.5,180.25) -- cycle ;
\draw  [dash pattern={on 4.5pt off 4.5pt}]  (170.5,130) -- (480.43,130.3) ;
\draw  [dash pattern={on 4.5pt off 4.5pt}]  (170.5,160) -- (480.43,160.3) ;
\draw  [dash pattern={on 4.5pt off 4.5pt}]  (170.5,100) -- (480.43,100.3) ;
\draw [color={rgb, 255:red, 45; green, 35; blue, 235 }  ,draw opacity=1 ]   (250.67,99.7) -- (280.83,130.93) ;
\draw [shift={(262.14,111.58)}, rotate = 46] [color={rgb, 255:red, 45; green, 35; blue, 235 }  ,draw opacity=1 ][line width=0.75]    (7.65,-2.3) .. controls (4.86,-0.97) and (2.31,-0.21) .. (0,0) .. controls (2.31,0.21) and (4.86,0.98) .. (7.65,2.3)   ;
\draw [color={rgb, 255:red, 45; green, 35; blue, 235 }  ,draw opacity=1 ]   (220.67,129.7) -- (250.83,160.93) ;
\draw [shift={(232.14,141.58)}, rotate = 46] [color={rgb, 255:red, 45; green, 35; blue, 235 }  ,draw opacity=1 ][line width=0.75]    (7.65,-2.3) .. controls (4.86,-0.97) and (2.31,-0.21) .. (0,0) .. controls (2.31,0.21) and (4.86,0.98) .. (7.65,2.3)   ;
\draw [color={rgb, 255:red, 45; green, 35; blue, 235 }  ,draw opacity=1 ]   (250.67,99.7) -- (260.17,130.27) ;
\draw [shift={(253.87,110.02)}, rotate = 72.73] [color={rgb, 255:red, 45; green, 35; blue, 235 }  ,draw opacity=1 ][line width=0.75]    (7.65,-2.3) .. controls (4.86,-0.97) and (2.31,-0.21) .. (0,0) .. controls (2.31,0.21) and (4.86,0.98) .. (7.65,2.3)   ;
\draw [color={rgb, 255:red, 45; green, 35; blue, 235 }  ,draw opacity=1 ]   (240.67,129.7) -- (250.17,160.27) ;
\draw [shift={(243.87,140.02)}, rotate = 72.73] [color={rgb, 255:red, 45; green, 35; blue, 235 }  ,draw opacity=1 ][line width=0.75]    (7.65,-2.3) .. controls (4.86,-0.97) and (2.31,-0.21) .. (0,0) .. controls (2.31,0.21) and (4.86,0.98) .. (7.65,2.3)   ;
\draw [color={rgb, 255:red, 45; green, 35; blue, 235 }  ,draw opacity=1 ]   (240.67,129.7) -- (250.67,99.7) ;
\draw [shift={(246.99,110.72)}, rotate = 108.43] [color={rgb, 255:red, 45; green, 35; blue, 235 }  ,draw opacity=1 ][line width=0.75]    (7.65,-2.3) .. controls (4.86,-0.97) and (2.31,-0.21) .. (0,0) .. controls (2.31,0.21) and (4.86,0.98) .. (7.65,2.3)   ;
\draw [color={rgb, 255:red, 45; green, 35; blue, 235 }  ,draw opacity=1 ]   (250.67,159.7) -- (260.67,129.7) ;
\draw [shift={(256.99,140.72)}, rotate = 108.43] [color={rgb, 255:red, 45; green, 35; blue, 235 }  ,draw opacity=1 ][line width=0.75]    (7.65,-2.3) .. controls (4.86,-0.97) and (2.31,-0.21) .. (0,0) .. controls (2.31,0.21) and (4.86,0.98) .. (7.65,2.3)   ;
\draw [color={rgb, 255:red, 208; green, 2; blue, 27 }  ,draw opacity=1 ]   (280.83,130.93) -- (400.83,159.6) ;
\draw [shift={(335.78,144.06)}, rotate = 13.44] [color={rgb, 255:red, 208; green, 2; blue, 27 }  ,draw opacity=1 ][line width=0.75]    (7.65,-2.3) .. controls (4.86,-0.97) and (2.31,-0.21) .. (0,0) .. controls (2.31,0.21) and (4.86,0.98) .. (7.65,2.3)   ;
\draw [color={rgb, 255:red, 208; green, 2; blue, 27 }  ,draw opacity=1 ]   (260.17,130.27) .. controls (280.7,150.67) and (320.03,150.67) .. (400.83,159.6) ;
\draw [shift={(323.87,151.48)}, rotate = 6.83] [color={rgb, 255:red, 208; green, 2; blue, 27 }  ,draw opacity=1 ][line width=0.75]    (7.65,-2.3) .. controls (4.86,-0.97) and (2.31,-0.21) .. (0,0) .. controls (2.31,0.21) and (4.86,0.98) .. (7.65,2.3)   ;
\draw  [color={rgb, 255:red, 0; green, 0; blue, 0 }  ,draw opacity=1 ][fill={rgb, 255:red, 0; green, 0; blue, 0 }  ,fill opacity=1 ] (247,100.13) .. controls (247,98.21) and (248.55,96.66) .. (250.46,96.66) .. controls (252.38,96.66) and (253.93,98.21) .. (253.93,100.13) .. controls (253.93,102.04) and (252.38,103.59) .. (250.46,103.59) .. controls (248.55,103.59) and (247,102.04) .. (247,100.13) -- cycle ;
\draw  [color={rgb, 255:red, 0; green, 0; blue, 0 }  ,draw opacity=1 ][fill={rgb, 255:red, 0; green, 0; blue, 0 }  ,fill opacity=1 ] (277,130.13) .. controls (277,128.21) and (278.55,126.66) .. (280.46,126.66) .. controls (282.38,126.66) and (283.93,128.21) .. (283.93,130.13) .. controls (283.93,132.04) and (282.38,133.59) .. (280.46,133.59) .. controls (278.55,133.59) and (277,132.04) .. (277,130.13) -- cycle ;
\draw  [color={rgb, 255:red, 0; green, 0; blue, 0 }  ,draw opacity=1 ][fill={rgb, 255:red, 0; green, 0; blue, 0 }  ,fill opacity=1 ] (257,130.13) .. controls (257,128.21) and (258.55,126.66) .. (260.46,126.66) .. controls (262.38,126.66) and (263.93,128.21) .. (263.93,130.13) .. controls (263.93,132.04) and (262.38,133.59) .. (260.46,133.59) .. controls (258.55,133.59) and (257,132.04) .. (257,130.13) -- cycle ;
\draw  [color={rgb, 255:red, 0; green, 0; blue, 0 }  ,draw opacity=1 ][fill={rgb, 255:red, 0; green, 0; blue, 0 }  ,fill opacity=1 ] (237,130.13) .. controls (237,128.21) and (238.55,126.66) .. (240.46,126.66) .. controls (242.38,126.66) and (243.93,128.21) .. (243.93,130.13) .. controls (243.93,132.04) and (242.38,133.59) .. (240.46,133.59) .. controls (238.55,133.59) and (237,132.04) .. (237,130.13) -- cycle ;
\draw  [color={rgb, 255:red, 0; green, 0; blue, 0 }  ,draw opacity=1 ][fill={rgb, 255:red, 0; green, 0; blue, 0 }  ,fill opacity=1 ] (217,130.13) .. controls (217,128.21) and (218.55,126.66) .. (220.46,126.66) .. controls (222.38,126.66) and (223.93,128.21) .. (223.93,130.13) .. controls (223.93,132.04) and (222.38,133.59) .. (220.46,133.59) .. controls (218.55,133.59) and (217,132.04) .. (217,130.13) -- cycle ;
\draw  [color={rgb, 255:red, 0; green, 0; blue, 0 }  ,draw opacity=1 ][fill={rgb, 255:red, 0; green, 0; blue, 0 }  ,fill opacity=1 ] (397,160.13) .. controls (397,158.21) and (398.55,156.66) .. (400.46,156.66) .. controls (402.38,156.66) and (403.93,158.21) .. (403.93,160.13) .. controls (403.93,162.04) and (402.38,163.59) .. (400.46,163.59) .. controls (398.55,163.59) and (397,162.04) .. (397,160.13) -- cycle ;
\draw  [color={blue}  ,draw opacity=1 ] (194.41,74) -- (307.15,74) -- (307.15,187.1) -- (194.41,187.1) -- cycle ;
\draw  [color={rgb, 255:red, 100; green, 100; blue, 100 }  ,draw opacity=1 ][dash pattern={on 0.84pt off 2.51pt}] (180.05,55.65) -- (469.8,55.65) -- (469.8,205.15) -- (180.05,205.15) -- cycle ;
\draw  [color={rgb, 255:red, 0; green, 0; blue, 0 }  ,draw opacity=1 ][fill={rgb, 255:red, 0; green, 0; blue, 0 }  ,fill opacity=1 ] (247,160.13) .. controls (247,158.21) and (248.55,156.66) .. (250.46,156.66) .. controls (252.38,156.66) and (253.93,158.21) .. (253.93,160.13) .. controls (253.93,162.04) and (252.38,163.59) .. (250.46,163.59) .. controls (248.55,163.59) and (247,162.04) .. (247,160.13) -- cycle ;

\draw (154.67,91) node [anchor=north west][inner sep=0.75pt]   [align=left] {$\displaystyle 1$};
\draw (154.67,121) node [anchor=north west][inner sep=0.75pt]   [align=left] {$\displaystyle 0$};
\draw (144.67,151) node [anchor=north west][inner sep=0.75pt]   [align=left] {$\displaystyle -1$};
\draw (240.67,208) node [anchor=north west][inner sep=0.75pt]   [align=left] {$\displaystyle ( -1,-1)$};
\draw (401.17,208) node [anchor=north west][inner sep=0.75pt]   [align=left] {$\displaystyle ( -1,1)$};
\draw (260,55.75) node [anchor=north west][inner sep=0.75pt]  [font=\scriptsize,color={blue}  ,opacity=1 ] [align=left] {$\displaystyle {\mathscr{C}}_{u}^{-1}( B_{L})$};
\draw (427.5,36.75) node [anchor=north west][inner sep=0.75pt]  [font=\scriptsize,color={black}  ,opacity=1 ] [align=left] {$\displaystyle {\mathscr{C}}_{Kh}^{-1}( B_{L})$};

\end{tikzpicture}
    \caption{Khovanov chain complex and skein flow category}
    \label{fig:example-maps}    
\end{figure}

\bibliographystyle{amsalpha}
\bibliography{refs}

\end{document}